\documentclass[10pt,twocolumn,a4paper]{IEEEtran}
\IEEEoverridecommandlockouts                     
\usepackage{cite,graphicx,amsmath,amssymb,amsthm}
\usepackage{subfigure,color}

\usepackage{fancyhdr}
\usepackage{color}
\usepackage{dsfont,url}

\newtheorem{theorem}{Theorem}

\newtheorem{lemma}{Lemma}

\newtheorem{corollary}{Corollary}

\newtheorem{proposition}{Proposition}

\newtheorem{conjecture}{Conjecture}

\newtheorem{sketch}{Sketch of Proof}

\newtheorem{definition}{Definition}

\newtheorem{remark}{Remark}

\newtheorem{example}{Example}

\newcommand{\defeq}{\stackrel{\Delta}{=}}

\newcommand{\bLam}{\mathbf{\Lambda}}

\begin{document}

\title{
Analysis and Design of Multiple-Antenna Cognitive Radios with Multiple Primary User Signals}

\author{David Morales-Jimenez, Raymond H.\ Y.\ Louie, Matthew R. McKay, and Yang Chen
\thanks{The work of D. Morales-Jimenez and M. R. McKay was supported by the Hong Kong Research Grants Council under grant number 616713. R. H. Y. Louie was supported by HKUST research grant IGN13EG02. Y. Chen was supported by research grant FDCT 077/2012/A3. Part of this work will be presented at the IEEE International Conference on Communications (ICC), Sydney (Australia), June 2014.}
\thanks{D. Morales-Jimenez, R. H. Y. Louie, and M. R. McKay are with the Department of Electronic and Computer Engineering, Hong Kong University of Science and Technology, Clear Water Bay, Kowloon (Honk Kong). (e-mail:\{eedmorales,eeraylouie,eemckay\}@ust.hk)}
\thanks{Y. Chen is with the Department of Mathematics, University of Macau, Av. Padre Tom\'{a}s Pereira, Taipa (Macau). (e-mail: yayangchen@umac.mo)}
\thanks{This work has been submitted to the IEEE for possible publication. Copyright may be transferred without notice, after which this version may no longer be accessible.}
}

 \maketitle 
\begin{abstract}
We consider multiple-antenna signal detection of primary user transmission signals by a secondary user receiver in cognitive radio networks. The optimal detector is analyzed for the scenario where the number of primary user signals is no less than the number of receive antennas at the secondary user. We first derive exact expressions for the moments of the generalized likelihood ratio test (GLRT) statistic, yielding approximations for the false alarm and detection probabilities. We then show that the normalized GLRT statistic converges in distribution to a Gaussian random variable when the number of antennas and observations grow large at the same rate. Further, using results from large random matrix theory, we derive expressions to compute the detection probability without explicit knowledge of the channel, and then particularize these expressions for two scenarios of practical interest: 1) a single primary user sending spatially multiplexed signals, and 2) multiple spatially distributed primary users. Our analytical results are finally used to obtain simple design rules for the signal detection threshold.
\end{abstract}



\section{Introduction}

Cognitive radio is a promising technology which can be used to improve the utilization efficiency of the radio spectrum by allowing secondary user (SU)
networks to co-exist with primary user (PU) networks through
spectrum sharing \cite{Mitola1999,haykin05,Akyildiz2008,Cabric2008,Peha2009}. A key requirement is that SU transmission will not adversely affect the PUs' performance. To achieve this, a common technique involves the SUs first detecting if at least one PU is transmitting, which is commonly referred to as ``spectrum sensing". If no signals are detected, the SUs are allowed to transmit. The importance of signal detection can be seen by its inclusion in the IEEE 802.22 standard, built on cognitive radio techniques \cite{Cordeiro2005}.

Signal detection has been extensively investigated over the past few decades (see  \cite{Mauchly1940,pillai1971,Wax1985} as examples of some seminal works), within different contexts of application (see, e.g., \cite{Cochran1995,Sirianunpiboon2013} in radar). Inspired by some of those seminal works, a number of signal detection tests have been proposed to detect PU transmission when there are multiple receive antennas at the SUs (see e.g., \cite{Besson2006,Zeng2008,
Zeng2009,Taherpour2010,Wang2010,fujikoshi2010,Ramirez2011,Wei2012b,Sedighi2013}). Optimality is often considered in the Neyman-Pearson sense, which involves comparing the generalized likelihood ratio (GLR) to a user-designed detection threshold. The GLR can be used to determine the false alarm and detection probabilities, which can then be subsequently used to design the threshold. The particular form of the GLR is dependent on the number of PU transmission signals, and whether noise and/or channel information is known at the SU receiver performing the signal detection. Albeit being a well-investigated subject, multiple-antenna based signal detection is a problem raising a substantial interest in the recent literature (see, e.g., \cite{Sedighi2013,Ramirez2011,Bianchi2011,Wei2012b,Taherpour2010}) since fundamental issues still remain open. In particular, very little is known for the performance of GLR-based detectors under the presence of multiple PU signals.

A reasonable scenario is to assume that nothing is known at the SU receiver, i.e., no noise and channel information are known. For this scenario, the false alarm and detection probability have been analyzed when there is only one PU signal (see e.g.,\ \cite{Wang2010,Taherpour2010}). However, the simultaneous presence of multiple PU signals is a common occurrence in current and next generation systems. This may occur, for example, in multiple-antenna systems where spatial multiplexing techniques are employed, or where multiple independent PUs (e.g., from adjacent cells) simultaneously access the same frequency channel. Furthermore, the number of PU signals are expected to grow given the current trend towards more dense networks with more users simultaneously served \cite{Li2011}. On the other hand, low-complexity cognitive devices (mobile units) can be reasonably assumed to have less antennas than the transmitters in the primary system. Thus, for many practical scenarios of interest, the number of PU signals $k$ is no less than the number of receive antennas at the SUs $n$. As a concrete example, we could have a primary system built on multiple-antenna base stations, where a large number of signals (say, e.g., $k=16$) are simultaneously transmitted in the downlink to different PUs. In contrast, we can consider mobile SUs having a limited number of antennas of, e.g., $n=4$. The non-symmetric complexity of base stations (with fixed deployments) and mobile terminals is indeed a common occurrence in practice which motivates our interest in cases where $k \geq n$. See, for instance, the LTE standard \cite{3GPPlte} or references on multiple-antenna based multi-user communications (e.g., \cite{choi04} and references therein).

From a practical perspective, however, the optimal test (i.e., optimal detector) takes a form which depends on the number of PU signals $k$ \cite{Ramirez2011}, which is generally not known to the cognitive devices. Therefore, a universally optimal detector would entail, prior to signal detection, the estimation of $k$ in order to determine the optimal test to be performed during the detection phase. Interestingly, for scenarios where $k \geq n$, the optimal test, referred to as the \textit{sphericity test} \cite{Mauchly1940}, takes the same form regardless of $k$. Therefore, the sphericity test would uniformly yield optimal results provided that the situation $k \geq n$ (i.e., full-rank signal) holds. This situation can be either estimated via, e.g., minimum description length (MDL) methods \cite{Huang2007}, or anticipated (assumed) in some typical scenarios where the number of PU signals can be known (or at least lower bounded) beforehand. This is the case of TV broadcasting stations or cellular base stations which, in compliance with wireless standards, have a given (known) number of transmitting antennas. Examples range from broadcasting standards, such as the European DVB-T2 \cite{Etsi} which considers two-antenna space-time Alamouti codes, to point-to-multipoint standards, such as IEEE 802.11n \cite{IEEE802b}, IEEE 802.16 \cite{IEEE802c}, or LTE \cite{3GPPlte}, which support up to sixteen transmit antennas according to their latest releases.

 
For these scenarios where $k \geq n$, exact expressions for the false alarm probability and the detection probability were derived in \cite{Wei2012b} when there are two receive antennas. For more general scenarios with arbitrary number of receive antennas and observations, \cite{Ramirez2011} conducted Monte Carlo simulations while \cite{pillai1971} \cite[pp. 230]{fujikoshi2010} derived infinite series expansions. However, the series expansions in \cite{pillai1971} involved complicated zonal polynomials or Meijer-G functions which are generally hard to compute (they are in fact integrals that need to be evaluated by numerical integration methods), while the false alarm probability expression in \cite[pp. 230]{fujikoshi2010} was not amenable to analysis. For the same general scenario, an approximation was considered in \cite{Wei2012b}; however, the approximation therein was only justified for the false alarm probability, and only then for a very small number of antennas. Despite having made some progress in the detection of multiple PU signals, the aforementioned works do not provide a tractable analysis for the probabilities of detection and false alarm to a full extent. It is our aim to fill this gap by providing accurate approximations for these probabilities which result in simple design rules for practical detectors.

In this paper, we derive accurate approximations for the false alarm and detection probabilities of the GLR detector\footnote{Note that the performance of the GLR detector has been previously shown to perform better than other detectors in many practical scenarios \cite{Wei2012b}, and thus we do not consider such comparisons in this paper.}
which: (i) are valid for any number of receive antennas, provided that $k \geq n$, and (ii) are easy-to-compute involving only a finite number of terms comprising the well-known Gamma function. This is facilitated by an expression for the moments of the GLR test (GLRT) statistic which we derive. Despite the computational benefits of these expressions over previous results, our results allow us to further analyze the detection performance in the asymptotic regime where the number of receive antennas and observations are large and of similar order. For this scenario, we first derive simple and accurate approximations for the moments and cumulants of the GLRT statistic, and then show that this statistic converges in distribution to a Gaussian random variable under the hypothesis of no PU signals being present. Moreover, we analyze the detection probability for a large number of PU signals with $k \geq n$. Using results from large random matrix theory, we show that the (instantaneous) detection probability can be accurately approximated without explicit knowledge of the channel for a practical number of antennas. Leveraging our analytical results, we then propose simple design rules to approximate the detection threshold that achieves a desired false alarm probability while maximizing the detection probability.

\section{Problem Statement}
\label{sec:statement}

Consider a wireless communications system where a SU receiver equipped with $n$ antennas is tasked with determining if PU transmission signals are present from $m$ independent and identically distributed (IID) observation sample vectors $\mathbf{x}_1,\ldots,\mathbf{x}_m$, where\footnote{$\mathbf{0}_{p,q}$ denotes the $p \times q$ matrix of all zeros.} $\mathbf{x}_\ell \sim \mathcal{CN}_{n,1}(\mathbf{0}_{n,1},\mathbf{R})$ for $\ell=1,\ldots,m$, and $\mathbf{R}$ is a $n \times n$ population covariance matrix. The $\ell$th sample vector $\mathbf{x}_{\ell}$ for this hypothesis testing problem is modeled as
\begin{align}\label{eq:single_detect}
& \mathcal{H}_0: \quad \mathbf{x}_{\ell} = \mathbf{n}_\ell \hspace{2cm} \text{no signal present} \notag \\
& \mathcal{H}_1: \quad \mathbf{x}_{\ell} = \mathbf{H} \mathbf{s}_\ell + \mathbf{n}_\ell \hspace{1cm} \text{signals present} 
\end{align}
where $\mathbf{n}_\ell \sim \mathcal{CN}_{n,1}(\mathbf{0}_{n,1}, \mathbf{I}_n N_0)$ denotes additive white Gaussian noise with variance $N_0$, $\mathbf{s}_{\ell} \in \mathbb{C}^k$ is the signal vector with ${\rm E} [ \mathbf{s}_{\ell} \mathbf{s}_{\ell}^\dagger ] = \mathbf{I}_k$, $\mathbf{H} \in \mathbb{C}^{n \times k}$ is the channel matrix from the PUs to the SU detector\footnote{Note that the PUs' transmit signal power is included in $\mathbf{H}$.}, which is assumed to be constant during the $m$ observation time periods, and $k$ is the number of PU transmission signals. Both $\mathbf{n}_\ell$ and $\mathbf{s}_{\ell}$ are assumed IID over $\ell=1,\ldots,m$, implying that the observation sample vectors $\mathbf{x}_{\ell}$ are also IID. Unless otherwise specified, we do not assume a specific distribution for $\mathbf{H}$; thus, our results can account for each PU transmission signal having different transmit power. We assume that $\mathbf{H}$, $k$ and $N_0$ are unknown at the detector, and that $\mathbf{H} \mathbf{H}^\dagger$ is positive-definite (full rank), i.e., $k \ge n$. The latter condition can correspond to the scenario where there are at least $n$ single-antenna transmitting PUs, or if there is at least one transmitting PU equipped with at least $n$ antennas which are utilized for spatial multiplexing. As discussed in the previous section, the full-rank condition can be either estimated prior to detection via, e.g., MDL methods, or anticipated in many typical scenarios with a known or lower bounded number of PU signals.

The detection problem in (\ref{eq:single_detect}) is equivalent to testing if the population covariance matrix $\mathbf{R}$ is one of two structures:
\begin{align}\label{eq:sample_detect}
& \mathcal{H}_0: \quad \mathbf{R}  = \mathbf{I}_n N_0 \hspace{2cm} \text{no signal present} \notag \\
& \mathcal{H}_1: \quad \mathbf{R} = \mathbf{H} \mathbf{H}^\dagger + \mathbf{I}_n N_0 \hspace{1cm} \text{signals present}  \; .
\end{align}
To proceed, it is convenient to introduce the observed data matrix $\mathbf{X}=\left[\mathbf{x}_1,\ldots,\mathbf{x}_m \right]$ and the sample covariance matrix
\begin{align} \notag
\mathbf{\hat{\mathbf{R}}} = \frac{1}{m} \sum_{\ell=1}^m \mathbf{x}_{\ell} \mathbf{x}_{\ell}^\dagger = \frac{1}{m} \mathbf{X} \mathbf{X}^\dagger \; .
\end{align}

Different testing criteria can be considered for the detection problem in (\ref{eq:single_detect}). Bayesian tests such as \cite{Sedighi2013,Sirianunpiboon2013} aim at minimizing an average cost function which involves all possible incorrect decisions. In contrast, we consider in this paper detectors of the Neyman--Pearson type, which aim at maximizing the probability of correct detection (probability of choosing $\mathcal{H}_1$ given $\mathcal{H}_1$) under a certain constraint on the maximum admissible probability of false alarm (probability of choosing $\mathcal{H}_1$ given $\mathcal{H}_0$).      
Since there are unknown parameters under both hypotheses ($N_0$ and $\mathbf{H}$), the Neyman--Pearson detector, which yields a uniformly most
powerful test, is not directly implementable. Therefore, we adopt the classical GLR approach as it has been shown to result in simple detectors with good performance \cite{Mardia1979}. The likelihood function of the observation samples is given by their joint density, i.e.,
\begin{align} \notag
f \left(\mathbf{x}_1,\ldots,\mathbf{x}_m\biggr| \mathbf{R} \right) = \frac{1}{\pi^{mn} \det(\mathbf{R})^m} {\rm etr}\left(-m \mathbf{\hat{R}} \mathbf{R}^{-1}   \right) ,
\end{align}
with ${\rm etr}(\cdot)=e^{{\rm Tr}(\cdot)}$.
The GLR $\mathcal{L}$, used for determining $\mathcal{H}_0$ or $\mathcal{H}_1$, is the ratio between likelihoods under $\mathcal{H}_0$ and $\mathcal{H}_1$, with these likelihoods maximized over the unknown parameters \cite[Sec. III]{Ramirez2011} \cite{Mardia1979}, i.e., 
\begin{align}
\mathcal{L} = \frac{ \sup\limits_{N_0 \in \mathbb{R}^{+}} f \left(\mathbf{x}_1,\ldots,\mathbf{x}_m\biggr| \mathbf{R}= \mathbf{I}_n N_0 \right)}{\sup\limits_{N_0 \in \mathbb{R}^{+}, ~ \mathbf{H} \in \mathbb{C}^{n \times k}} f \left(\mathbf{x}_1,\ldots,\mathbf{x}_m\biggr| \mathbf{R}= \mathbf{H}\mathbf{H}^\dagger+\mathbf{I}_n N_0 \right)} .
\end{align}

The GLRT then determines $\mathcal{H}_0$ or $\mathcal{H}_1$ by testing whether $\mathcal{L}$ is above or under a user-specified detection threshold.
The GLR $\mathcal{L}$ has been studied in the literature under different assumptions on the rank of $\mathbf{H}\mathbf{H}^\dagger$; see e.g., \cite{Bianchi2011} for rank-1 or \cite{Wax1985,Ramirez2011} for more general assumptions. Here, we are concerned with the case of $\mathbf{H}\mathbf{H}^\dagger$ having full rank, which has yet to be studied in detail in the literature. In this case, $\mathcal{L}$ can be obtained explicitly and the GLRT yields the well-known \textit{sphericity test} \cite{Mauchly1940}. We use $W$ to denote the corresponding GLRT statistic which admits\footnote{The GLRT statistic usually presented in literature is $\frac{1}{W}$, which is used to form the sphericity test \cite{Mauchly1940}. However, we work with $W$ for mathematical convenience.} \cite[Sec. III]{Ramirez2011}
 \begin{align}\label{eq:test}
W \triangleq   \frac{\frac{{\rm Tr}\left( \mathbf{X} \mathbf{X}^\dagger\right)}{n}}{ {\rm det}(\mathbf{X} \mathbf{X}^\dagger)^{\frac{1}{n}}} \mathop{\gtrless}\limits_{\mathcal{H}_0}^{\mathcal{H}_1} \eta ,
 \end{align}
where $\eta$ is a user-specified detection threshold.
  Thus, PU signals are deemed to be present if $W > \eta$, while no PU transmission is deemed if $W \leq \eta$. Note that $\mathcal{H}_1$, as defined in (\ref{eq:single_detect}), implies the presence of an arbitrary number of PU signals, i.e., it does not necessarily imply the presence of $k \geq n$ signals. However, if the full-rank assumption is violated ($k < n$), the test in (\ref{eq:test}) will only yield sub-optimal results without the probability of detection being maximized.

 \subsection{False Alarm and Detection Probability} \label{FA_PD_defs}
To evaluate the performance of the GLRT statistic (\ref{eq:test}), we consider the false alarm and the detection probability. The false alarm probability is
 \begin{align}\label{eq:falsealarm}
 {\rm P}_{\rm FA}(\eta) & \defeq {\rm Pr}\left(W_0 > \eta\right) = 1 - {\rm F}_{W_0}(\eta)
 \end{align}
where 
\begin{align} \notag
W_0\defeq\frac{\frac{{\rm Tr}\left( \mathbf{X} \mathbf{X}^\dagger\right)}{n}}{ {\rm det}(\mathbf{X} \mathbf{X}^\dagger)^{\frac{1}{n}}} \; ,\quad \mathbf{X} \sim \mathcal{CN}_{n,m}\left(\mathbf{0}_{n,m}, \mathbf{I}_n N_0 \right) 
\end{align}
and ${\rm F}_{W_0}(\eta)$ denotes the cumulative distribution function (c.d.f.) of $W_0$.
 The threshold  $\eta$ is typically chosen to ensure the false alarm probability does not exceed a maximum value $\alpha_0 \in (0,1)$, i.e.,
 \begin{align} \notag
 \eta = {\rm P}_{\rm FA}^{-1}(\alpha_0) \; .
 \end{align}
Note that the false alarm probability does not depend on $\mathbf{H}$ and, thus, the detection threshold can be designed regardless of the channel statistics; for instance, the designed threshold will be independent of the PUs' transmit power.
 
The probability of correct detection is 
   \begin{align} \label{eq:detection}
   {\rm P}_{\rm D}(\eta) & \defeq {\rm Pr}\left(W_1 > \eta\right) = 1 - {\rm F}_{W_1}(\eta)
   \end{align}
  where
  \begin{align} \notag
  W_1\defeq\frac{\frac{{\rm Tr}\left( \mathbf{X} \mathbf{X}^\dagger\right)}{n}}{ {\rm det}(\mathbf{X} \mathbf{X}^\dagger)^{\frac{1}{n}}} \; , \quad \mathbf{X} \sim \mathcal{CN}_{n,m}\left(\mathbf{0}_{n,m}, \mathbf{H} \mathbf{H}^\dagger + \mathbf{I}_n N_0  \right) 
  \end{align}
  and ${\rm F}_{W_1}(\eta)$ denotes the c.d.f.\ of $W_1$.
  

   
   

\section{C.d.f. of $W_0$ and $W_1$: Non-Asymptotic Analysis}
\label{sec:non-asymp}

In this section, we derive expressions for the c.d.f.\ of $W_0$ and $W_1$ for arbitrary $n$ and $m$ with $k \geq n$. We first present closed-form expressions for the moments of $W_0$ and $W_1$.

\subsection{Exact Moments}

\begin{theorem}\label{the:semicorr_fullrank}
The $p$th ($p \in \mathbb{Z}^{+}$) moment of $W_0$ and $W_1$, for $p < n(m-n+1)$, are respectively given by
\begin{align} \label{eq:iid_fullrank}
\mu_{W_0,p}&={\rm E}\left[W_0^p\right] \notag \\
&= \frac{\Gamma\left(mn\right) }{n^p \Gamma\left(mn -p \right)} \prod_{j=0}^{n-1} \frac{ \Gamma\left(m-n+1-\frac{p}{n} + j \right)}{\Gamma\left(m-n+1+j \right)} , \\
\label{eq:semicorr_fullrank}
\mu_{W_1,p}&={\rm E}\left[W_1^p\right] \notag \\
&= \frac{ p! \prod_{i=1}^n y_i^{-\frac{p}{n}} }{n^p } \prod_{j=0}^{n-1}\frac{ \Gamma\left(m-n+1-\frac{p}{n}+j\right)  }{ \Gamma\left(m-n+1+j \right)} \notag \\
& \times \sum_{k_1+\ldots + k_n=p} \, \prod_{i=1}^n \frac{\Gamma\left(m-\frac{p}{n}+k_i\right) y_i^{k_i}}{\Gamma(k_i+1) \Gamma\left(m-\frac{p}{n}\right)} ,
\end{align}
where $\Gamma(\cdot)$ denotes the Gamma function,
$k_1,\ldots,k_n$ are nonnegative integers,
and $N_0<y_1 \le y_2 \le \ldots\le y_n < \infty$ denote the eigenvalues of $\mathbf{H} \mathbf{H}^\dagger + \mathbf{I}_n N_0$.
\end{theorem}
\begin{proof}
See Appendix \ref{app:semicorr_fullrank}.
\end{proof}
Note that (\ref{eq:iid_fullrank}) has been derived previously in \cite{Wei2012b}, while (\ref{eq:semicorr_fullrank}) is new.

\emph{Condition $p<n(m-n+1)$:} The condition on $p$, required for (\ref{eq:iid_fullrank}) and (\ref{eq:semicorr_fullrank}) to hold, becomes milder as the difference between the numbers of observations and antennas grow. The first few moments can be obtained for even a very small number of antennas and observations. For example, for $p=3$, the condition is satisfied when $n=2$ and $m=3$.

\subsection{C.D.F. Approximation: Edgeworth Expansion}

Armed with the derived expressions for the moments, we now aim at characterizing the c.d.f. of $W_0$ and $W_1$ in an easy-to-compute form, which can help in understanding the performance of the GLRT in order to design the detection threshold $\eta$.

As observed in the conference version of this paper \cite{ray2014}, the empirical distribution of $W_0$ approaches a Gaussian as the number of antennas and observations grow large with fixed ratio. As also pointed out in \cite{ray2014}, a similar phenomenon is observed for $W_1$.
This convergence motivates us to consider a Gaussian approximation for the c.d.f.\ of $W_0$ and $W_1$, corrected with additional terms obtained by the Edgeworth expansion \cite{blinnikov1997,li2013}. Specifically, the c.d.f. of an arbitrary random variable $X$ in $L$-truncated Edgeworth expansion takes the form \cite[eq. (45)]{li2013}
\begin{align} \label{edgeworth}
{\rm F}_X(x) & \approx \Phi(\tilde{x}) \nonumber \\
 & \hspace{-5mm} - \frac{e^{-\frac{\tilde{x}^2}{2}}}{\sqrt{2\pi}} \sum_{s=1}^L \sum_{\{\mathcal{J}_s\}} \frac{{\rm{He}}_{s+2r}(\tilde{x})}{\sigma_X^{s+2r}} \prod_{\ell=1}^s \frac{1}{j_{\ell}!} \left( \frac{\kappa_{X, \ell +2}}{(\ell +2)!} \right)^{j_{\ell}}
\end{align}
where $\tilde{x}=\frac{x-{\rm E}\left[X\right]}{\sigma_X}$, $\sigma_X$ is the standard deviation of $X$, $\Phi(\cdot)$ is the c.d.f. of a standardized Gaussian, $\{\mathcal{J}_s\}$ is the set containing the nonnegative integer solutions to $j_1+2j_2+\ldots+sj_s=s$, and $r=j_1+j_2+\ldots+j_s$.
Further, $\kappa_{X, p}$ is the $p$th cumulant of $X$, related to the first $p$ moments through
\begin{align} \label{cumulants_moments}
\kappa_{X,1} &= \mu_{X,1} \notag \\
\kappa_{X,p} &= \mu_{X,p} - \sum_{\ell=1}^{p-1} \binom{p-1}{\ell-1} \kappa_{X,\ell} \mu_{X, p-\ell} \; , \quad p \geq 2,
\end{align}
with $\mu_{X, p} = {\rm E} \left[X^p\right]$,
and ${\rm He}_{\ell}(z)$ is the Chebyshev--Hermite polynomial \cite[eq. (13)]{blinnikov1997}
\begin{align} \notag
{\rm He}_{\ell}(z) = \ell! \sum_{k=0}^{\lfloor\frac{\ell}{2}\rfloor} \frac{ (-1)^k z^{\ell-2k}}{k! (\ell-2k)! 2^k} \; ,
\end{align}
where $\lfloor \cdot \rfloor$ denotes the floor function.

In (\ref{edgeworth}), a truncation limit $L$ implies that $\kappa_{X, p}$, $p=3,4,\ldots,L+2$, are involved in the corrected c.d.f.
Particularizing (\ref{edgeworth}) for $L=2$ results in the following simple approximation for the c.d.f. of $W_0$ and $W_1$:
\begin{align}\label{eq:cdf_W01}
{\rm F}_{W_\ell}(\eta) \approx \mathcal{G}\left(\frac{\eta- \mu_{W_\ell,1}}{\sigma_{W_\ell}} ; \sigma_{W_\ell}, \kappa_{W_\ell,3}, \kappa_{W_\ell,4}\right),
\end{align}
for $\ell = 0,1$, with
\begin{align}\label{eq:Gfunction}
&\mathcal{G}(x; \sigma, \kappa_{3}, \kappa_{4}) = \Phi(x) - \sqrt{\frac{2}{\pi}} \frac{ e^{-\frac{x^2}{2}}}{12 \sigma^3}\left(\kappa_{3} (x^2-1)   \right. \notag \\
&  \left. +\frac{\kappa_{4}}{4 \sigma} x(x^2-3)+ \frac{(\kappa_{3})^2}{12 \sigma^3}x \left(x^4-10x^2+15\right) \right) .
\end{align}
More terms can be added ($L>2$) with an expected increase in accuracy; however, with $L=2$, i.e., involving up to the fourth cumulant, the c.d.f. of $W_0$ and of $W_1$ are already approximated with high accuracy.
This can be observed in Figs.\ \ref{fig:false_alarm_cdf} and \ref{fig:correct_detection_cdf}, which plot respectively the probability of false alarm $ {\rm P}_{\rm FA}(\eta) $ and the probability of detection $ {\rm P}_{\rm D}(\eta) $, both as a function of the detection threshold $\eta$. The `Analytical (Gaussian)' curves correspond to a Gaussian approximation without any correction terms ($L=0$), i.e.,  $\rm{F}_{W_1}(\eta) \approx \Phi((\eta-\mu_{W_1,1})/{\sigma_{W_1}})$ and $\rm{F}_{W_0}(\eta) \approx \Phi\left((\eta-\mu_{W_0,1})/{\sigma_{W_0}}\right)$, while the `Analytical (Correction)' curves are plotted using (\ref{eq:cdf_W01}). The `Analytical (Beta) \cite{Wei2012b}' curves are plotted using the Beta function approximation introduced in \cite{Wei2012b}. These results are all compared with the true c.d.f., computed via Monte Carlo simulation\footnote{To simulate the detection probability, we generate $\mathbf{H}$ as $\sim \mathcal{CN}_{n,k}\left(\mathbf{0}_{n,k}, \mathbf{I}_n\right)$, which is held constant for $m$ observation periods. This corresponds to a scenario where a single PU transmits $k$ spatially multiplexed signals and where the channel undergoes Rayleigh fading.}.

For the false alarm probability curves in Fig.\ \ref{fig:false_alarm_cdf}, we observe that the Gaussian approximation deviates from Monte Carlo simulations, thus justifying the use of additional correction terms. With these terms, the `Analytical (Correction)' curve closely matches the simulations with improved accuracy as $n$ and $m$ increase. Finally, the `Analytical (Beta) \cite{Wei2012b}' curve shows a satisfactory agreement for $\{n=4, m=15\}$, but it starts deviating for larger number of antennas and observations. 

For the detection probability curves in Fig.\ \ref{fig:correct_detection_cdf}, we again observe a significant deviation of the Gaussian approximation from the Monte Carlo simulations, especially for the case $\{n=4, m=15\}$. Moreover, for $\{n=10, m=20\}$, the `Analytical (Beta) \cite{Wei2012b}' curves are inaccurate for most detection probabilities as opposed to our `Analytical (Correction)' curves, which closely match the simulations for both configurations.

\begin{figure}[t]
\centerline{\includegraphics[width=0.85\columnwidth]{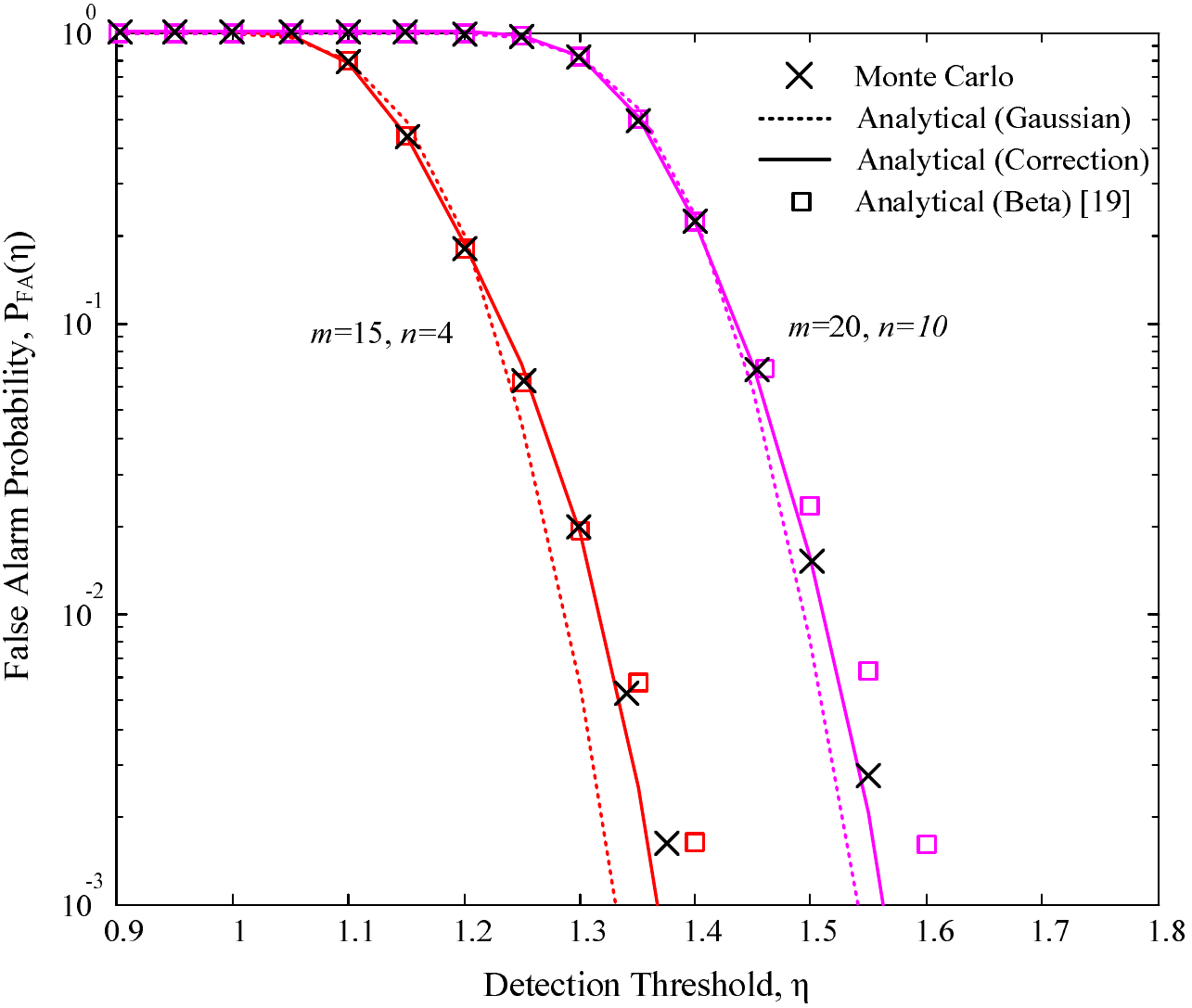}}
\caption{Probability of false alarm vs.\ detection threshold, with $k=n$.}
\label{fig:false_alarm_cdf}
\end{figure}

\begin{figure}[t]
\centerline{\includegraphics[width=0.85\columnwidth]{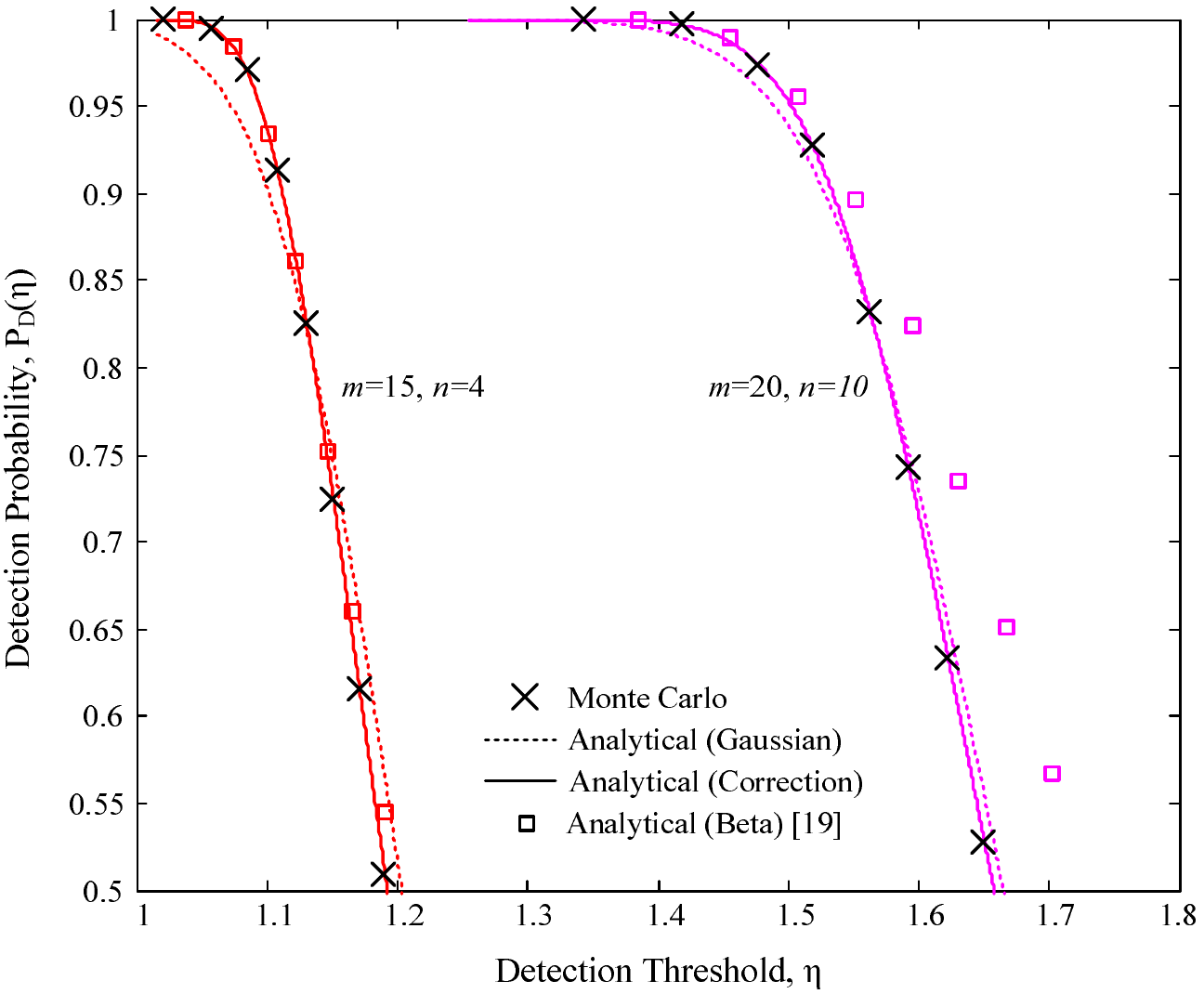}}
\caption{Probability of detection vs.\ detection threshold, with $N_0=5$ and $k=n$.}
\label{fig:correct_detection_cdf}
\end{figure}
\section{Asymptotic Analysis}
\label{sec:asympAnalysis}

The false alarm and detection probabilities can be calculated for arbitrary number of antennas $n$ (provided that $k \geq n$) and observations $m$ using the moment expressions (\ref{eq:iid_fullrank}) and (\ref{eq:semicorr_fullrank}). However, the computation of such expressions, involving $n$-products of Gamma functions, gets rather involved when $n$ and $m$ are large. We are thus motivated in this section to look into more convenient asymptotic expressions for the moments and cumulants, which allow in turn for an efficient computation of the c.d.f. with (moderately) large $n$ and $m$.




\subsection{Moments and Cumulants of $W_0$}
We aim to obtain simple expressions for the cumulants of $W_0$ which, plugged into (\ref{edgeworth}), allow for an efficient computation of the false alarm probability.
Letting $n$ and $m$ be large but finite, we first provide an asymptotic expansion for the moments.

\begin{proposition}\label{lem:expansion_iid}
The $p$th ($p \in \mathbb{Z}^{+}$) moment of $W_0$, with $p < n(m-n+1)$ and $c=\frac{n}{m} \in (0,1)$, admits the expansion
\begin{align}\label{eq:logmellin_iid}
\mu_{W_0,p} = \sum_{j=0}^\infty \frac{ \beta_{p,j}(c)}{n^{2j}}
\end{align}
with $\beta_{p,0}(c) = e^{A_{p,0}(c)}$, and for $j >0$,
\begin{align} \label{muExpansionCoeffs}
\beta_{p,j}(c) =  e^{A_{p,0}(c)} \hspace{-0.5cm} \sum_{i_1+2 i_2+ \ldots +j i_j =j}  \hspace{0.1cm}  \prod_{r=1}^j   \frac{ A_{p,r}(c)^{i_r} } {i_r!}   \; ,
\end{align}
where the sum is taken over the non-negative integer solutions to $i_1+2 i_2 +\ldots+j i_j=j$, and the coefficients $A_{p,q}(c)$ are 
\begin{align} \label{a0}
A_{p,0}(c) &= p-p\left(1-\frac{1}{c}\right) \log(1-c) , \\ \label{a1}
A_{p,1}(c) &= \frac{-p^2}{2}\log(1-c)+\frac{c p(12-11c+6(1-p)(c-1))}{12(c-1)},
\end{align}
and, for $q>1$, 
\begin{align} \label{aq}
A_{p,q}(c) &= \frac{(c p)^q \left(c+12 p-12 c p-c q^2 \right)}{12 c q \left(-1+q^2\right)(1-c)^q } \nonumber \\
&+ \frac{(c p)^q \left(-12 p+c \left(-1+q^2\right)\right) }{12 c q \left(-1+q^2\right)} -\frac{c^q}{q} \, H_{p,-q} \nonumber \\ 
& + \sum _{j=1}^{q-1} \frac{B_{2j+2} \, c^{2j} \, (2j)_{q-j} \, (p c)^{q-j}} {4j \, (j+1) \, (q-j)!} \left(1-(1-c)^{-j-q} \right) ,
\end{align}
where $H_{a,b}$ and $B_k$ are, respectively, the Harmonic numbers and Bernoulli numbers \cite{Abramowitz1970}.
\end{proposition}
\begin{proof}
See Appendix \ref{app:asymp_moments}.
\end{proof}


We may now obtain corresponding expansions for the cumulants.

\begin{proposition} \label{lem:cumulant_expansion_iid}
The $p$th ($p \in \mathbb{Z}^{+}$) cumulant of $W_0$, with $p < n(m-n+1)$ and $c=\frac{n}{m} \in (0,1)$, admits the expansion
\begin{align} \label{kappaExpansion}
\kappa_{W_0,p} = \sum_{j=0}^{\infty} \frac{\alpha_{p,j}(c)}{n^{2(p-1+j)}}   \; ,
\end{align} 
where
\begin{align} \label{alphaRecursion}
\alpha_{p,j}(c) &= \beta_{p,p-1+j}(c) \notag \\
& - \sum_{r=1}^{p-1} \binom{p-1}{r-1} \cdot  \sum_{\ell=0}^{p-r+j} \alpha_{r,\ell}(c) \, \beta_{p-r, \, p-r+j-\ell}(c)   \; ,
\end{align}
with $\alpha_{1,j}(c) = \beta_{1,j}(c)$ and $\beta_{p,j}(c)$ given by (\ref{muExpansionCoeffs}) with $\beta_{p,0}(c) = e^{A_{p,0}(c)}$.
\end{proposition}
\begin{proof}
Follows 
by substituting (\ref{eq:logmellin_iid}) into (\ref{cumulants_moments}) and rearranging terms in the resultant series.  
\end{proof}


For large (but finite) $n$ and $m$, and $c=\frac{n}{m} \in (0,1)$, the leading-order term $\alpha_{p,0}(c)$ dominates the $p$th cumulant, giving
\begin{align} \label{kappaApprox}
\kappa_{W_0,p} \approx \frac{\alpha_{p,0}(c)}{n^{2(p-1)}}.
\end{align}
Using (\ref{alphaRecursion}) to obtain $\alpha_{p,0}(c)$ for the first four cumulants yields
\begin{align} \label{k1}
\kappa_{W_0,1} & \approx a_1^{-a_3} \, e \\ \label{k2} 
\kappa_{W_0,2} & \approx \frac{-e^2}{n^2} a_1^{-2a_3}  \, \left[c + a_2 \right] \\
\label{k3}
\kappa_{W_0,3} & \approx \frac{-e^3}{n^4} a_1^{-3a_3-1} \, 
 \left[ c^2 (2c-3)- 6 c  a_1 a_2 - 3 a_1 a_2^2 \right] \\
\label{k4}
\kappa_{W_0,4} & \approx \frac{-e^4}{n^6} a_1^{-4a_3-2} \,   \left[ c^3 (16 + c (6c-23)) \right. \nonumber \\
& \left. - 12 a_1 c^2 (3c-4) a_2 +  48 c  a_1^2 a_2^2 + 16 a_1^2 a_2^3 \right] ,
\end{align}
where $a_1=1-c$, $a_2=\ln{a_1}$, and $a_3=1-\frac{1}{c}$.

Note that the leading-order cumulant approximations in (\ref{k1})-(\ref{k4}) are much simpler to compute than the exact cumulants,  especially when $n$ is large. The false alarm probability can thus be efficiently computed via the approximated cumulants by plugging (\ref{k1})-(\ref{k4}) into (\ref{eq:cdf_W01}). To see the accuracy of such an approximation, we compare it with the same c.d.f. expression (\ref{eq:cdf_W01}) computed with the exact cumulants. This comparison is shown in Fig. \ref{fig:false_alarm_cdf_asymptotic}, where we observe that the difference between the ``asymptotic cumulants" and the ``exact cumulants" curves is indistinguishable even for $n$ as low as 4.

\begin{figure}[t]
\centerline{\includegraphics[width=0.85\columnwidth]{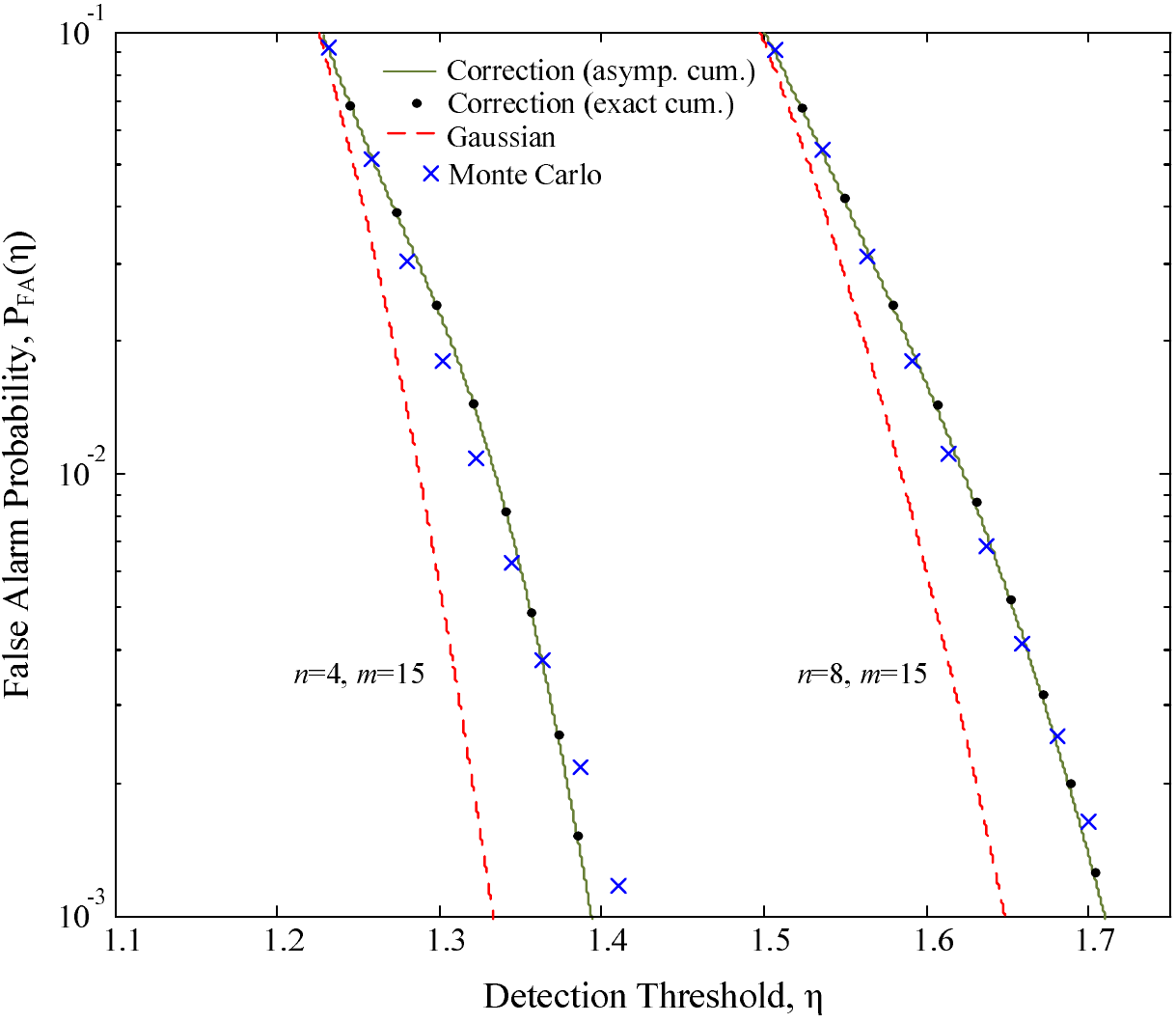}}
\caption{Probability of false alarm vs.\ detection threshold, with $k=n$.}
\label{fig:false_alarm_cdf_asymptotic}
\end{figure}


\begin{remark}[The case $c=1$]
Note that the expansions for the moments and the cumulants, as given in Propositions \ref{lem:expansion_iid} and \ref{lem:cumulant_expansion_iid}, are not valid when $c=1$. For this particular case, different expansions involving all powers of $n^{-1}$ are obtained in Appendix \ref{app:asymp_moments_c1} and provide analytic continuation to the aforementioned propositions. Remarkably, the leading-order approximation to the $p$th cumulant, $p < n(m-n+1)$, when $c=1$ is found to be
\begin{align} \label{cumulants_c1}
\kappa_{W_0,1} & \approx e \\
\kappa_{W_0,2} & \approx \frac{1}{n^2} e^2 \, (\mathcal{C}+\log{n}) \\
\kappa_{W_0,p} & \approx \frac{1}{n^p} e^p \, (p-1)! \, \zeta(p-1) , \quad p>2,
\end{align}
where $\mathcal{C}$ is the Euler-Gamma constant \cite[(8.367)]{Gradshteyn1965} and $\zeta(\cdot)$ is the Riemann-Zeta function \cite[(9.5)]{Gradshteyn1965}. This result, albeit practically less meaningful than for the case $c \neq 1$, may still be of theoretical interest in other contexts. 
\end{remark}

\subsection{Gaussian Convergence of $W_0$}

We now show the asymptotic convergence of $W_0$ to a Gaussian distribution.

\begin{theorem}\label{the:asymptotic_cdf}
Let $n \to \infty$ with  $\frac{n}{m} \to \bar c \in (0,1)$. Then\footnote{$x \overset{d}{\to} y$ implies $x$ converges in distribution to $y$.} 
\begin{align}\label{eq:asymoptotic_iid}
n\left(W_0 - \bar{\mu}\right) \overset{d}{\to} \mathcal{N}\left(0, \bar{\sigma}^2 \right)
\end{align}
where
\begin{align}
\bar{\mu} 
&= e^{} (1-\bar c)^{\frac{1-\bar c}{\bar c}} ,
\end{align}
and
\begin{align}
\bar{\sigma}^2 
&=  e^{  2} (1- \bar c)^{\frac{2(1- \bar c)}{\bar c }} \left(\ln\left(\frac{1}{1- \bar c}\right)-\bar c \right) .
\end{align}
\end{theorem}
\begin{proof}
See Appendix \ref{app:asymptotic_cdf}.
\end{proof}

The result above confirms what was empirically observed in \cite{ray2014}, i.e., the distribution of $W_0$ approaches a Gaussian as the number of antennas $n$ and observations $m$ grow large with a fixed ratio. Motivating examples for the large $n$ regime can be found when considering a secondary system built on a massive antenna array setup, e.g., a secondary (non-licensed) base station with a very large number of antennas. On the other hand, there are applications where $m$ can be very large. For TV channels, for example, the IEEE 802.22 standard \cite{IEEE802} specifies the sensing estimation to be performed at speeds of under $1$ ms per channel, which along with a typical sampling frequency of $10$ MHz yield a maximum of $m = 10000$ samples to sense a given channel. Moreover, the cellular LTE standard defines a subframe as a transmission time interval (TTI) of $1$ ms \cite{Furht2009} and, therefore, the required sensing time should be well below the TTI (say, $1/10$ ms) in order to maximize the time available for data transmission. This results in $m = 1500$ samples considering a bandwidth configuration of $10$ MHz and sampling frequency of $15$ MHz.

\subsection{Asymptotic Analysis of $W_1$}
\label{sec:asymptoticW1}

Here, we first aim to obtain simple expressions for the cumulants of $W_1$ which, plugged into (\ref{edgeworth}), allow for an efficient computation of the detection probability for large (but finite) $n$ and $m$. Recall that $W_1$ is the GLRT statistic under the hypothesis of transmitted PU signals being present. In this case, as shown by (\ref{eq:semicorr_fullrank}), the moments (and therefore the cumulants) involve the eigenvalues of $\mathbf{H} \mathbf{H}^\dagger + \mathbf{I}_n N_0$. Hence, it is convenient to first define certain functions of these eigenvalues.

\begin{definition} \label{def:psi}
Let $\{y_1, y_2, \ldots, y_n\}$ denote the eigenvalues of $\mathbf{H} \mathbf{H}^\dagger + \mathbf{I}_n N_0$, with $\mathbf{H} \mathbf{H}^\dagger$ positive definite, then
\begin{align} \label{psi1_def}
\Psi_1 & \triangleq \prod_{i=1}^n y_i^{-\frac{1}{n}} \\ \label{psiL_def}
\Psi_{\ell + 1} & \triangleq \frac{1}{n}\sum_{i=1}^n y_i^\ell , \quad \ell \geq 1, 
\end{align}
with $\Psi_1 \in \left( 0, N_0^{-1} \right)$, $\Psi_{\ell + 1} \in \left( N_0,\infty \right)$ .
\end{definition}

In the next proposition, we provide simple expressions for the cumulants of $W_1$ for large $n$ and $m$. Note that $\mathbf{H}$ is assumed to be constant over the sensing time (i.e., constant over the $m$ samples), which may seem contradictory in the large $m$ regime. However, as discussed above, the sensing time (yielding very large values of $m$ in the aforementioned examples) is required to be well below the subframe duration, which is usually designed to be shorter than the typical channel coherence time.


\begin{proposition}\label{lemm:semi_variance}
For large (but finite) $n$ and $m$, and $c=\frac{n}{m} \in (0,1)$, the first three cumulants of $W_1$ are given by 
\begin{align} \label{eq:asympMeanW1}
\kappa_{W_1,1} &=  e (1-c)^{\frac{1-c}{c}}  \Psi_1 \Psi_2 +O\left(\frac{1}{n^2}\right) \\
\kappa_{W_1,2} &=  \frac{1}{n^2} \Psi_1^2  e^2 (1-c)^{2 \frac{1-c}{c}} \notag \\
&  \times \left[  \left(  \ln \left(\frac{1}{1-c}\right)-2c\right) \Psi_2^2 + c \Psi_3 \right] +O\left(\frac{1}{n^4}\right) \label{eq:varW1} \\
\kappa_{W_1,3} &= \frac{1}{n^4} \Psi_1^3  e^3 (1-c)^{3 \frac{1-c}{c}} \notag \\
& \hspace{-3mm} \times \left[  \left(\frac{c^2 (-10+9 c)}{-1+c}+3 \ln(1-c) (4 c+\ln(1-c))\right) \Psi_2^3 \right. \notag \\
& \hspace{-3mm} -\Psi_2\Psi_3 3c(3 c+2 \ln(1-c))+\Psi_4 2c^2 \bigg] +O\left(\frac{1}{n^6}\right), \label{eq:k3W1}
\end{align}
with $\Psi_{\ell}$ given in Definition \ref{def:psi}.
\end{proposition}
\begin{proof}
See Appendix \ref{app:semi_variance}.
\end{proof}
Although only the first three cumulants are given in Proposition \ref{lemm:semi_variance}, higher order cumulants can be derived by following the same approach. However, such derivations become more tedious as the cumulant order increases. Nevertheless, as shown later in this paper, the first three cumulants are enough for an accurate computation of the detection probability.

Plugged into (\ref{edgeworth}), these cumulants yield the approximate detection probability for a given PU-SU channel realization $\mathbf{H}$, which determines $\Psi_{\ell}$ in Proposition \ref{lemm:semi_variance}. In practice, $\mathbf{H}$ is a random communication channel, and thus the detection probability can be seen as a random function of $\mathbf{H}$. It turns out, however, that as $n$ and $k$ grow large, this function converges to a deterministic value, which depends only on some statistical properties of $\mathbf{H}$, but not on its specific distribution. 




\emph{Deterministic Equivalents for $\Psi_\ell$}:
Note from (\ref{psi1_def})-(\ref{psiL_def}) that
\begin{align} \label{eq:psiLL}
\Psi_1 &= {\rm det}^{-\frac{1}{n}}\left(\mathbf{H} \mathbf{H}^\dagger + \mathbf{I}_n N_0 \right) \; , \notag \\ 
\Psi_{\ell +1} &=\frac{1}{n} {\rm Tr}\left(\left(\mathbf{H} \mathbf{H}^\dagger + \mathbf{I}_n N_0 \right)^{\ell} \right) , \quad \ell \geq 1 ,
\end{align}
for which we aim to find deterministic equivalents $\bar\Psi_{\ell}$ such that\footnote{$x \overset{a.s.}{\to} y$ implies that $x$ converges almost surely to $y$.} 
\begin{align} \notag
| \Psi_\ell - \bar \Psi_\ell | \overset{a.s.}{\to} 0  \quad \mathrm{as} \,\,\,  n \to \infty  \,\,\, \mathrm{with} \,\,\, \frac{k}{n} \to \beta > 1,
\end{align}
where the deterministic equivalent $\bar\Psi_{\ell}$ does not depend on the PU-SU channel realization $\mathbf{H}$, but only on its statistical properties. To that end, we first state the following assumption. 

\emph{Assumption 1:}
The empirical distribution of the eigenvalues of an $n \times n$ Hermitian matrix $\mathbf{H H^{\dagger}}$, denoted by\footnote{$\mathbf{1}\{ E \}$ denotes the indicator function of an event $E$.} ${\rm{F}}_\mathbf{H H^{\dagger}}^{(n)}(x) = \frac{1}{n} \sum_{i=1}^n \mathbf{1}\{ y_i \leq x \}$, satisfies
\begin{align}\label{eq:H_limit}
{\rm{F}}_\mathbf{H H^{\dagger}}^{(n)}(x) \overset{a.s.}{\to} {\rm{F}}_{\beta}(x) \; ,\quad \forall x \in \mathbb{R} - \{0\}
\end{align}
as $n \to \infty$ with $ \frac{k}{n} \to \beta$, and ${\rm{F}}_{\beta}(x)$ commonly referred to as the asymptotic spectrum with density denoted by $f_{\beta}(x)$.

Under the above assumption, commonly adopted in large random matrix theory (see e.g.\ \cite{ledoit2013,tulino2004}), the limiting quantities $\bar\Psi_{\ell}$ are closely related to some asymptotic results which we invoke next.

Let us first connect $\Psi_1$ with the mutual information of the PU-SU channel through
\begin{align} \label{psi1_mi}
\Psi_1 &= \frac{1}{N_0} \exp{\left( -\frac{1}{n} \, \mathcal{I}_{\mathbf{H}}\left(\frac{1}{N_0}\right) \right)} ,
\end{align}
where
\begin{align} \label{mi}
\mathcal{I}_{\mathbf{H}} \left(\frac{1}{N_0}\right) = \log {\rm det} \left( \mathbf{I}_n + \frac{1}{N_0} \mathbf{H} \mathbf{H}^\dagger \right) 
\end{align}
is the mutual information of ${\mathbf{H}}$ with average SNR $1/N_0$. The asymptotics of this quantity have been extensively studied in information theory under different assumptions on the statistics of ${\mathbf{H}}$ (see, e.g., \cite{Moustakas2003,Chen2012,Kazakopoulos2011,tulino2004,Hachem2008}). In view of (\ref{psi1_mi}), the existing asymptotic results for (\ref{mi}) can be directly translated into the corresponding limiting value\footnote{Note that the limiting value of (\ref{mi}) yields the limiting value of (\ref{psi1_mi}) due to smoothness of the exponential function.} $\bar \Psi_1$.

Let us now turn our attention to $\Psi_{\ell +1}$, $\ell \geq 1$, which can be related to the so-called ``moments'' of $\mathbf{H} \mathbf{H}^\dagger$ as shown in the following lemma.

\begin{lemma}\label{lemm:psi}
The convergence values $\bar\Psi_{\ell +1}$, $\ell \geq 1$, are given by
\begin{align} \label{eq:psi}
\bar\Psi_{\ell +1} = N_0^{\ell} + \sum_{r=1}^\ell \binom{\ell}{r} N_0^{\ell -r} \mathcal{M}_r ,
\end{align} 
where $\mathcal{M}_r$ is the $r$th moment: 
\begin{align} \label{eq:Jr}
\mathcal{M}_r= \lim_{ \begin{subarray}{c} n \to \infty \\
\frac{k}{n} \to \beta \end{subarray}} \frac{1}{n} {\rm Tr} \left( \left( \mathbf{H} \mathbf{H}^\dagger \right)^r \right) .
\end{align} 
\end{lemma}
\begin{proof}
Follows straightforwardly from (\ref{eq:psiLL}) by using the binomial expansion. 
\end{proof}

Both the mutual information and $\mathcal{M}_r$
have been broadly studied for $n, k \to \infty$ in the context of information theory and large random matrix theory under different assumptions on the statistics of ${\mathbf{H}}$. The rich body of existing results allows for computing the limiting values $\bar \Psi_{\ell}$ for a broad number of scenarios. Among these results, we focus on two particular cases, especially relevant in our detection problem:

1) \textbf{IID case}: the entries of $\mathbf{H}$ are IID with zero mean and $\frac{1}{n}$ variance. This is the case for multiple PU signals being co-located as, e.g., when a single PU sends spatially multiplexed signals with equal transmit powers. In this case, the asymptotic value of (\ref{mi}) is given by \cite[Eq. (105)]{Moustakas2003} (see also, e.g., \cite[Eq. (95)]{Chen2012}, \cite[Eq. (1.14)]{tulino2004}), which yields
\begin{align} \label{eq:psi1}
\bar \Psi_1 &= \frac{1}{N_0} \left(1+\frac{1}{N_0}-\frac{1}{4} \mathcal{F}\left(\frac{1}{N_0},\beta \right) \right)^{-\beta}  \notag \\
& \times  \left(1+\frac{\beta}{N_0}-\frac{1}{4} \mathcal{F}\left(\frac{1}{N_0},\beta \right) \right)^{-1} \exp{ \left( \frac{N_0}{4} \mathcal{F}\left(\frac{1}{N_0},\beta \right) \right) }
\end{align}
with $\frac{k}{n} \to \beta$ as $n \to \infty$, and 
\begin{align}
\mathcal{F}\left( x,z \right) = \left( \sqrt{x (1+\sqrt{z})^2 +1 } - \sqrt{x (1-\sqrt{z})^2 +1 } \right)^2 .
\end{align}
Furthermore, $\mathcal{M}_r$
is obtained as \cite[Eq. (2.102)]{tulino2004}
\begin{align} \label{eq:momentsHiid}
\mathcal{M}_r^{(\rm{iid})} = \frac{1}{r} \sum_{i=1}^r \binom{r}{i} \binom{r}{i-1} \beta^i ,
\end{align}
which plugged into (\ref{eq:psi}) yield $\bar\Psi_{\ell +1}$, $\ell \geq 1$. We can now compute the limiting $\bar\Psi_2$, $\bar\Psi_3$, $\bar\Psi_4$, required for the cumulants in Proposition \ref{lemm:semi_variance}. Thus, setting $\ell = \{1,2,3\}$ in (\ref{eq:psi}) gives
\begin{align} \label{eq:psi2}
\bar\Psi_{2} &= \beta + N_0 , \\ \label{eq:psi3}
\bar\Psi_{3} &= \beta^2+\beta \left(1+2 N_0 \right)+ N_0^2 , \\ \label{eq:psi4}
\bar\Psi_{4} &= \beta^3+3 \beta^2 (1+N_0) + \beta \left(1+3 N_0+3 N_0^2\right)+N_0^3 .
\end{align}

2) \textbf{Unequal variances}: $\mathbf{H} = \mathbf{H}_{\rm{iid}} \mathbf{\Sigma}^{1/2}$ where $\mathbf{H}_{\rm{iid}}$ has IID entries with zero mean and $\frac{1}{n}$ variance, whilst $\mathbf{\Sigma}$ is a diagonal matrix with non-negative entries $\sigma_1^2,\ldots,\sigma_k^2$ uniformly bounded. This accommodates multiple signals from spatially distributed PUs having different transmit powers. For this case we invoke \cite[Thm. 1]{Hachem2008} to obtain
\begin{align} \label{eq:psi1_semi}
\bar \Psi_1 &= \frac{\beta \, \delta}{N_0}  \exp{ \left( \frac{\beta^2}{N_0}\delta \tilde\delta \right)} \prod _{i=1}^k \left(1+\frac{\beta}{N_0} \delta \sigma_i^2 \right)^{-\frac{1}{n}} ,
\end{align}
with $\frac{k}{n} \to \beta$,
\begin{align} \label{eq:delta}
\delta = \frac{1}{\beta} \left(1+\frac{\beta}{N_0} \tilde\delta\right)^{-1} ,
\end{align}
and $\tilde\delta$ being the unique positive solution to
\begin{align} \label{eq:deltatilde}
\tilde\delta = \frac{1}{k} \sum _{i=1}^k \frac{\sigma_i^2 (\beta \tilde\delta+N_0)}{\sigma_i^2+N_0+\beta \tilde\delta} .
\end{align}
For $\mathcal{M}_r$, according to \cite[Thm. 4]{Li2004},\footnote{See also \cite[Appx. C]{Morales2015} for a more general result which encompasses non-line-of-sight channels, i.e., $\mathbf{H}$ with non-zero mean entries.}
\begin{align} \label{eq:momentsHsemi}
\mathcal{M}_r^{(\rm{unequal})} &= \notag \\
& \hspace{-7mm} \sum_{i=1}^r \frac{\beta^i}{k^i} \sum_{r_1 +\ldots +r_i=r} \xi(r_1,\ldots,r_i) \cdot {\rm Tr}\left( \mathbf{\Sigma}^{r_1} \right) \cdots {\rm Tr}\left( \mathbf{\Sigma}^{r_i} \right) , 
\end{align}
where $\{r_1,\ldots,r_i\}$ are the strictly positive integer solutions to $r_1 +\ldots +r_i=r$ satisfying $r_1 \leq\ldots \leq r_i$, and
\begin{align} \label{eq:momentsHsemiCoeffs}
\xi(r_1,\ldots,r_i) = \frac{r!}{(r-i+1)! f_1! \cdots f_r!},
\end{align}
with $f_j$ being the number of entries in $\{r_1,\ldots,r_i\}$ equal to $j$. We can now compute $\bar\Psi_{\ell +1}$, $\ell \geq 1$, for the unequal variances case by plugging (\ref{eq:momentsHsemi}) in (\ref{eq:psi}), which results in
\begin{align} \label{eq:psi2_semi}
\bar\Psi_{2} &= N_0 + \frac{\beta}{k} {\rm Tr}(\mathbf{\Sigma})  , \\ \label{eq:psi3_semi}
\bar\Psi_{3} &= N_0^2 + 2 N_0 \frac{\beta}{k} {\rm Tr}(\mathbf{\Sigma}) + \frac{\beta}{k} {\rm Tr}(\mathbf{\Sigma}^2) + \frac{\beta^2}{k^2} {\rm Tr}^2(\mathbf{\Sigma}) , \\ \label{eq:psi4_semi}
\bar\Psi_{4} &= N_0^3 + 3 N_0^2 \frac{\beta}{k} {\rm Tr}(\mathbf{\Sigma}) +3 N_0 \left( \frac{\beta}{k} {\rm Tr}(\mathbf{\Sigma}^2) + \frac{\beta^2}{k^2} {\rm Tr}^2(\mathbf{\Sigma}) \right) \notag \\
& + \frac{\beta}{k} {\rm Tr}(\mathbf{\Sigma}^3) + 3 \frac{\beta^2}{k^2} {\rm Tr}(\mathbf{\Sigma}) {\rm Tr}(\mathbf{\Sigma}^2) +  \frac{\beta^3}{k^3} {\rm Tr}^3(\mathbf{\Sigma}) .
\end{align}

\begin{figure}[t]
\centerline{\includegraphics[width=0.85\columnwidth]{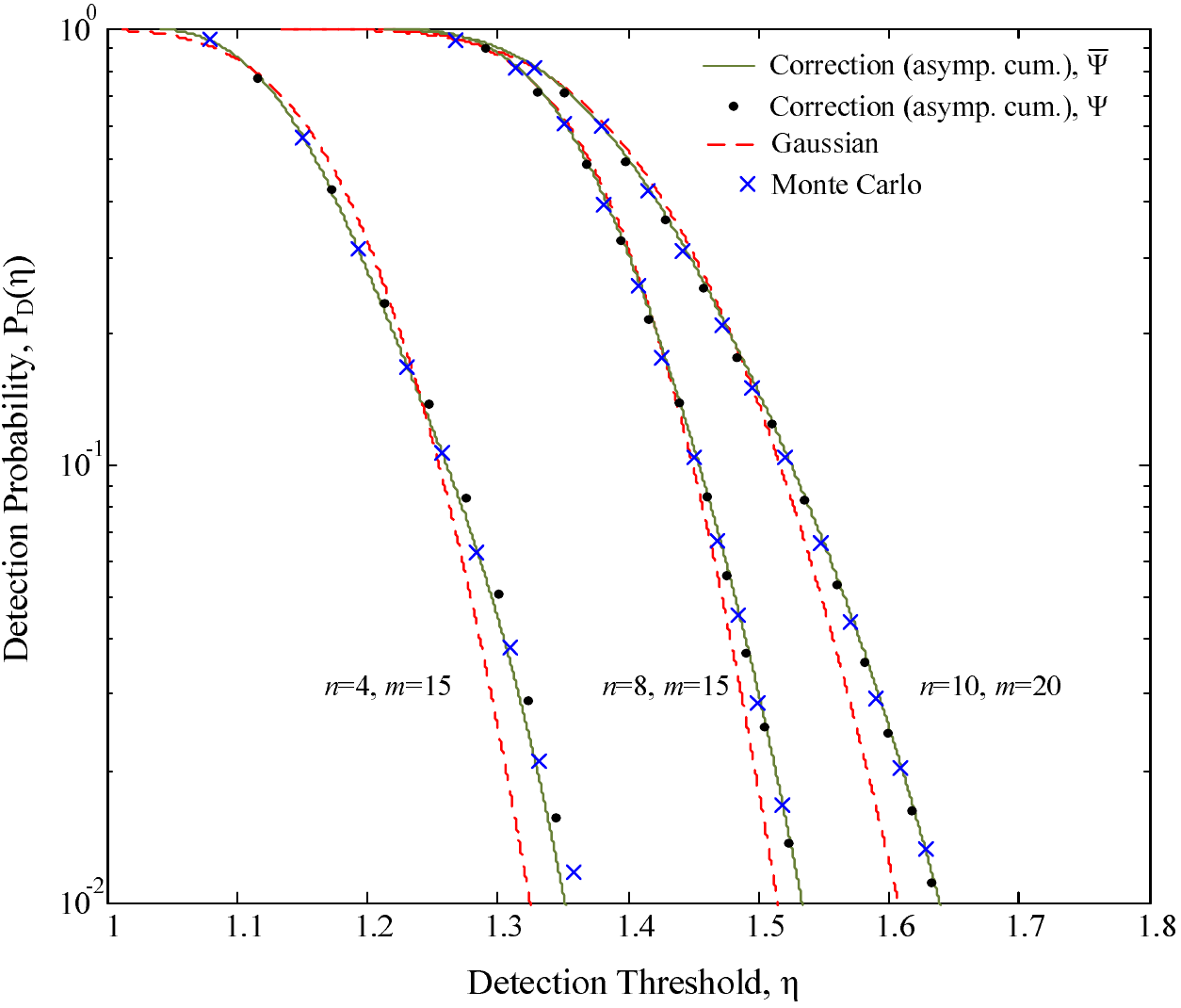}}
\caption{Detection probability vs.\ detection threshold, with $k=n$, $N_0=6.$}
\label{fig:detection_cdf_asymptotic}
\end{figure}

As we will see shortly, the convergence $| \Psi_\ell - \bar \Psi_\ell | \overset{a.s.}{\to} 0$ is quite fast, with $\bar \Psi_\ell$ giving an accurate approximation of $\Psi_\ell$ for not-so large $n$. Thus, under either the IID or the unequal variances assumption on $\mathbf{H}$, the cumulants in Proposition \ref{lemm:semi_variance} can be accurately approximated using $\bar\Psi_{\ell}$, $\ell =1,\ldots,4$, given respectively by either (\ref{eq:psi1}) and (\ref{eq:psi2})-(\ref{eq:psi4}), or (\ref{eq:psi1_semi}) and (\ref{eq:psi2_semi})-(\ref{eq:psi4_semi}). Plugging these cumulants in (\ref{edgeworth}), we can compute the detection probability in two typical scenarios of interest: 1) detection of a single PU which transmits $k$ spatially multiplexed signals with equal power, or 2) multiple spatially distributed PUs with different transmit powers. As an example, Fig. \ref{fig:detection_cdf_asymptotic} shows the detection probability vs. the detection threshold corresponding to the first scenario (IID case), with the same number of antennas at the PU and the SU, i.e., $k=n$. The solid curve is obtained with the asymptotic values $\bar\Psi_{\ell}$, while the dots are computed via the exact $\Psi_{\ell}$ for a particular channel realization $\mathbf{H}$. On the one hand, we see that the correction curve provides a satisfying match with Monte Carlo simulations. On the other hand, the correction based on $\bar\Psi_{\ell}$ is almost indistinguishable from the one based on $\Psi_{\ell}$ for a not-so-large\footnote{Larger values of $n$ could be shown with an increased accuracy. However, we deliberately considered a relatively low number of antennas to show the accuracy and quick convergence of our asymptotic results.} $n=8$, which shows the quick convergence of $\Psi_{\ell}$. Remarkably, even for a moderate number of antennas and observations, the detection probability can be computed without explicitly knowing $\mathbf{H}$.

\section{Threshold Design}
\label{sec:design}

\begin{figure}[t]
\centerline{\includegraphics[width=0.85\columnwidth]{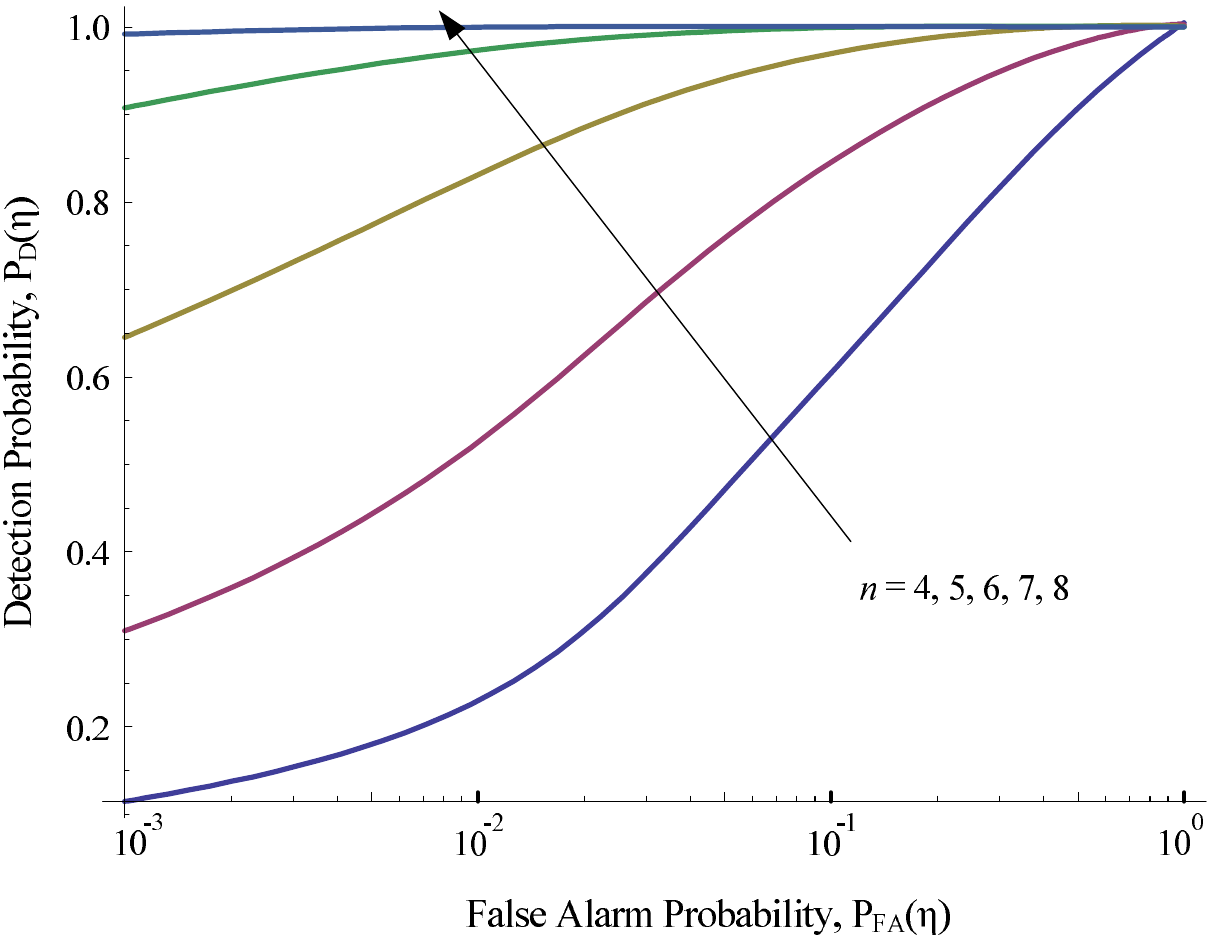}}
\caption{Analytic ROC curve (detection probability vs.\ false alarm probability) with $N_0=6$, $m = 5n$ and $k=n$.}
\label{fig:ROC_analytic}
\end{figure}

Having derived expressions to compute the probabilities of detection and false alarm, we now put these expressions to work in order to design the detection threshold, which will be used at the SUs to perform the GLR test, i.e., to determine the presence/absence of PU signals. As previously discussed, the PU-SU channel is typically unknown in practice, and therefore, the design usually relies on a false alarm probability requirement. Note that the false alarm probability is independent of $k$, and therefore, the number of PU signals is not required to design the threshold. Once the threshold is set, we can, under some statistical assumptions on $\mathbf{H}$, e.g., IID or with unequal variances, compute the corresponding detection probability for different PU-SU scenarios such as, e.g., a single transmitting PU with $k$ antennas, or $k$ spatially distributed PUs. Note that the probability of detection is a performance measure computed offline and that this quantity is not required at the SUs. Therefore, no specific knowledge on $k$ is required at the cognitive device in order to perform the GLR test.

\begin{figure}[t]
\centerline{\includegraphics[width=0.97\columnwidth]{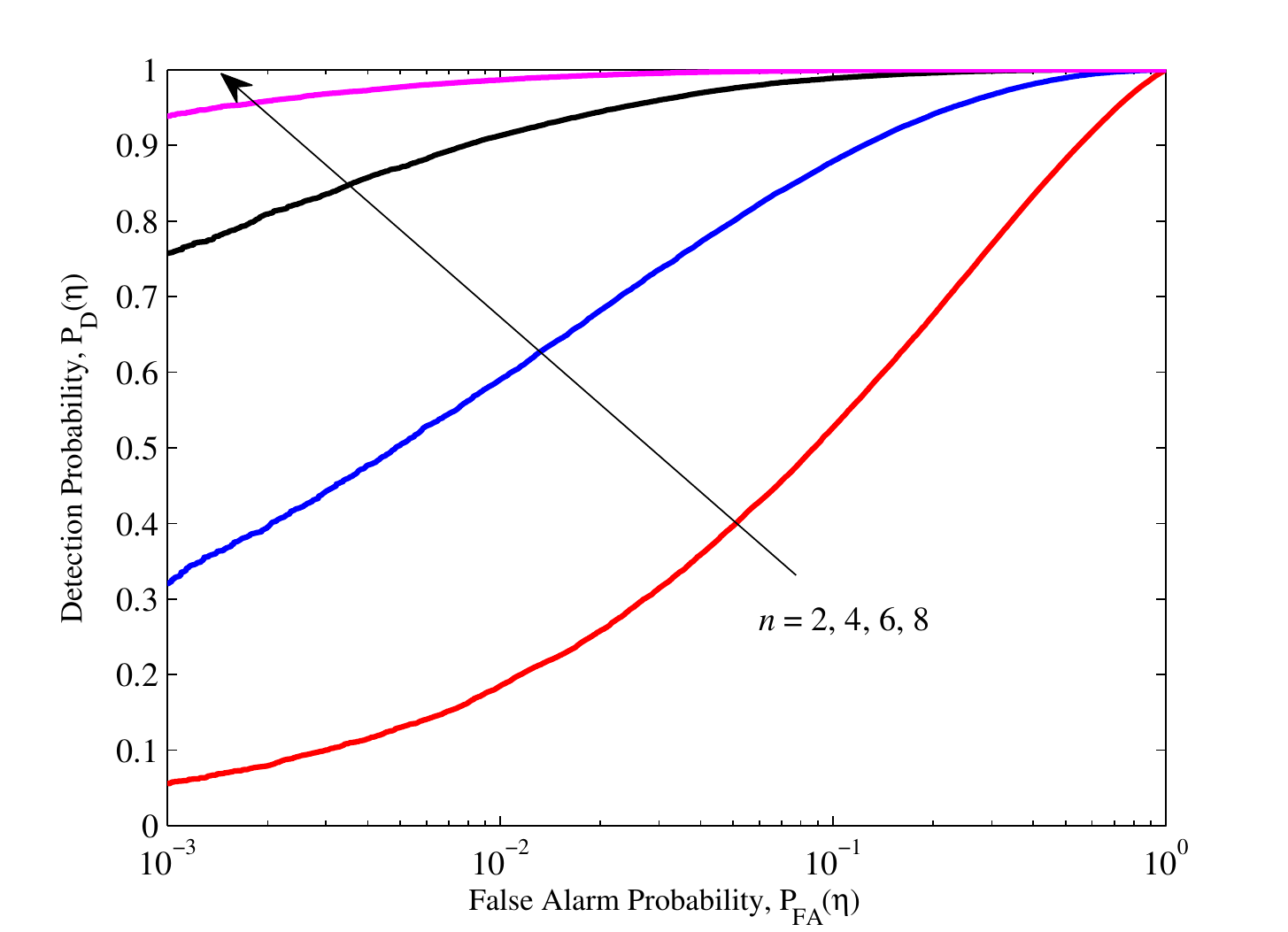}}
\caption{ROC curve (detection probability vs. false alarm probability) with $k=8$, $m=32$, $N_0=6$. }
\label{fig:roc_increasing_n}
\end{figure}

\begin{figure}[t]
\centerline{\includegraphics[width=0.97\columnwidth]{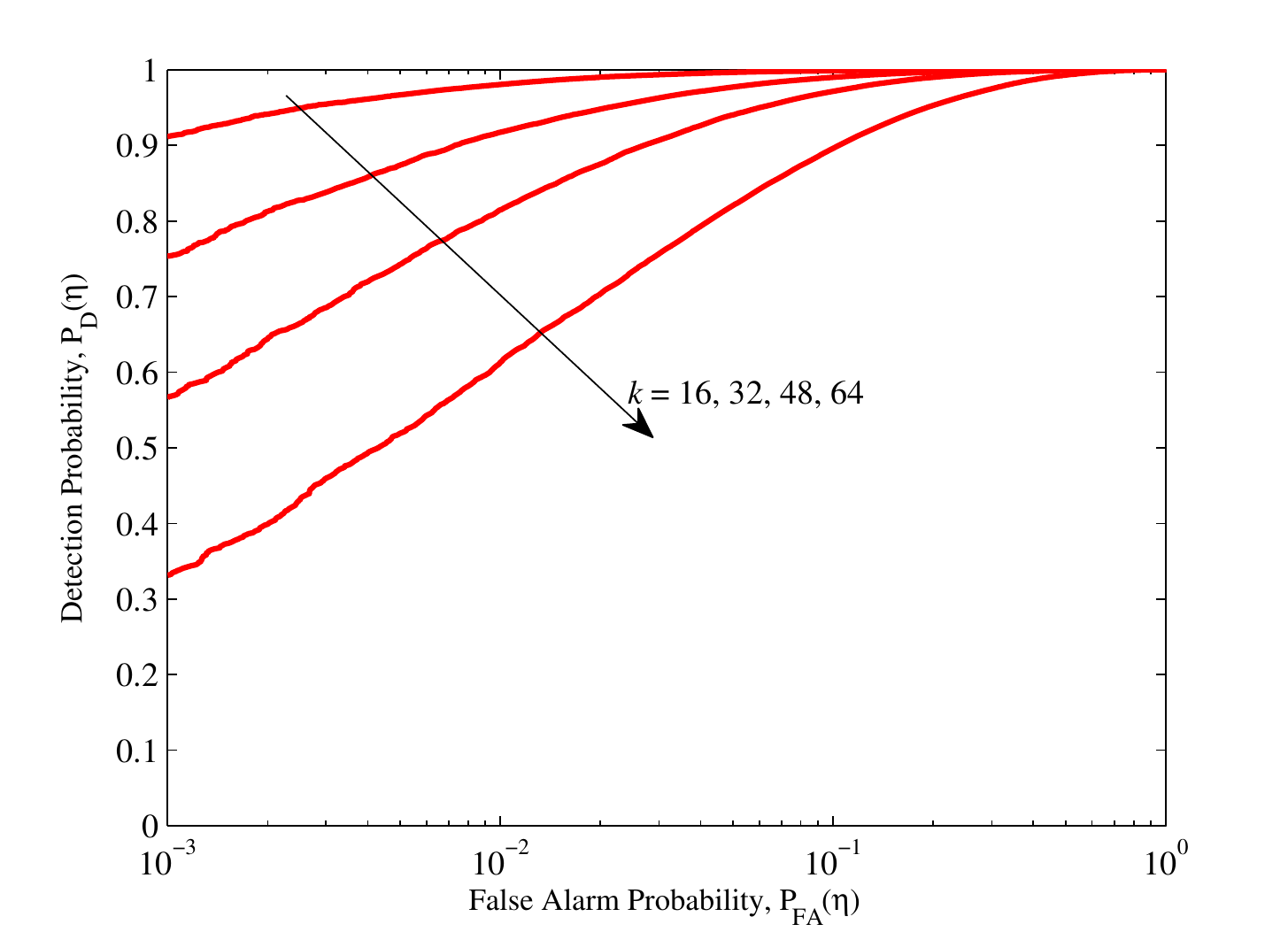}}
\caption{ROC curve (detection probability vs. false alarm probability) with $n=8$, $m=40$, $N_0=6$. }
\label{fig:roc_largeK}
\end{figure}

From (\ref{eq:cdf_W01}), the false alarm probability is approximated by
\begin{align}\label{eq:pfa_design}
 {\rm P}_{\rm FA}(\eta) &= 1 - {\rm F}_{W_0}(\eta)  \notag \\
 &\approx 1 - \mathcal{G}\left(\frac{\eta- \mu_{W_0,1} 	 }{\sigma_{W_0}} , \sigma_{W_0},\kappa_{W_0,3},\kappa_{W_0,4}\right),
\end{align}
where $\mu_{W_0,1}$, $\sigma_{W_0}^2$, $\kappa_{W_0,3}$, and $\kappa_{W_0,4}$ are given by (\ref{k1})-(\ref{k4}), and thus the minimum threshold which can satisfy a false alarm probability requirement of $\alpha_0$ can be approximated as
\begin{align}\label{eq:approx_threshold_total}
\eta_0 \approx  {\rm P}_{\rm FA}^{-1}(\alpha_0).
\end{align}
This threshold can be computed numerically from (\ref{eq:pfa_design}), and then used to obtain the corresponding detection probability
\begin{align}\label{eq:pd_design}
 {\rm P}_{\rm D}(\eta_0) &= 1 - {\rm F}_{W_1}(\eta_0)  \notag \\
 &\approx 1 - \Phi(x) + \sqrt{\frac{2}{\pi}} \frac{ e^{-\frac{x^2}{2}}}{12 \sigma_{W_1}^3} \kappa_{W_1,3} (x^2-1),  
\end{align}
with $x=\frac{\eta_0-\mu_{W_1,1}	 }{\sigma_{W_1}}$ and $\mu_{W_1,1}$, $\sigma_{W_1}^2$, $\kappa_{W_1,3}$ given by (\ref{eq:asympMeanW1})-(\ref{eq:k3W1}). The pair of values $\{ {\rm P}_{\rm FA}(\eta), {\rm P}_{\rm D}(\eta) \}$ defines the receive operating characteristics (ROC) curve, which is plotted in Fig. \ref{fig:ROC_analytic} for the IID scenario (a single transmitting PU with $k$ antennas), with $N_0=6$, $k=n$, and $m=5n$. We see that, with an average SNR of $-10 \log_{10} N_0 = -7.78$ dB, a low false alarm probability and high detection probability can be simultaneously achieved with $n>7$ antennas and $m>35$ observation samples. To better illustrate the effect of increasing $n$, we now keep the number of PU signals $k$ and observations $m$ fixed with $k=8$, $m=32$, and $N_0=6$. The corresponding ROC curve is shown in Fig. \ref{fig:roc_increasing_n}, where we observe that for a given false alarm probability, the detection probability increases as $n$ grows, implying that SUs with more antennas attain a higher accuracy in detecting the PU signals.

Let us now consider a primary system where a (possibly) large number $k$ of PU signals are simultaneously transmitted (via spatial multiplexing) in the downlink. In contrast, consider mobile SUs having a limited number of antennas $n=8$, representative of low-complexity cognitive radios. In this type of scenarios, we aim at understanding the implications of having a growing number of signals $k$. This is illustrated in Fig. \ref{fig:roc_largeK}, which shows the ROC curve for different values of $k$ with $n=8$, $m=40$, $N_0=6$. We observe that, for a given false alarm probability, the detection probability decreases as the number of PU signals $k$ grows.
In fact, it can be seen from the definitions of $W_0$ and $W_1$ in Section \ref{FA_PD_defs} that, in the limit $k \to \infty$, the presence of ``infinitely'' many transmitted signals would be indistinguishable from noise.

\begin{figure}[t]
\centerline{\includegraphics[width=0.85\columnwidth]{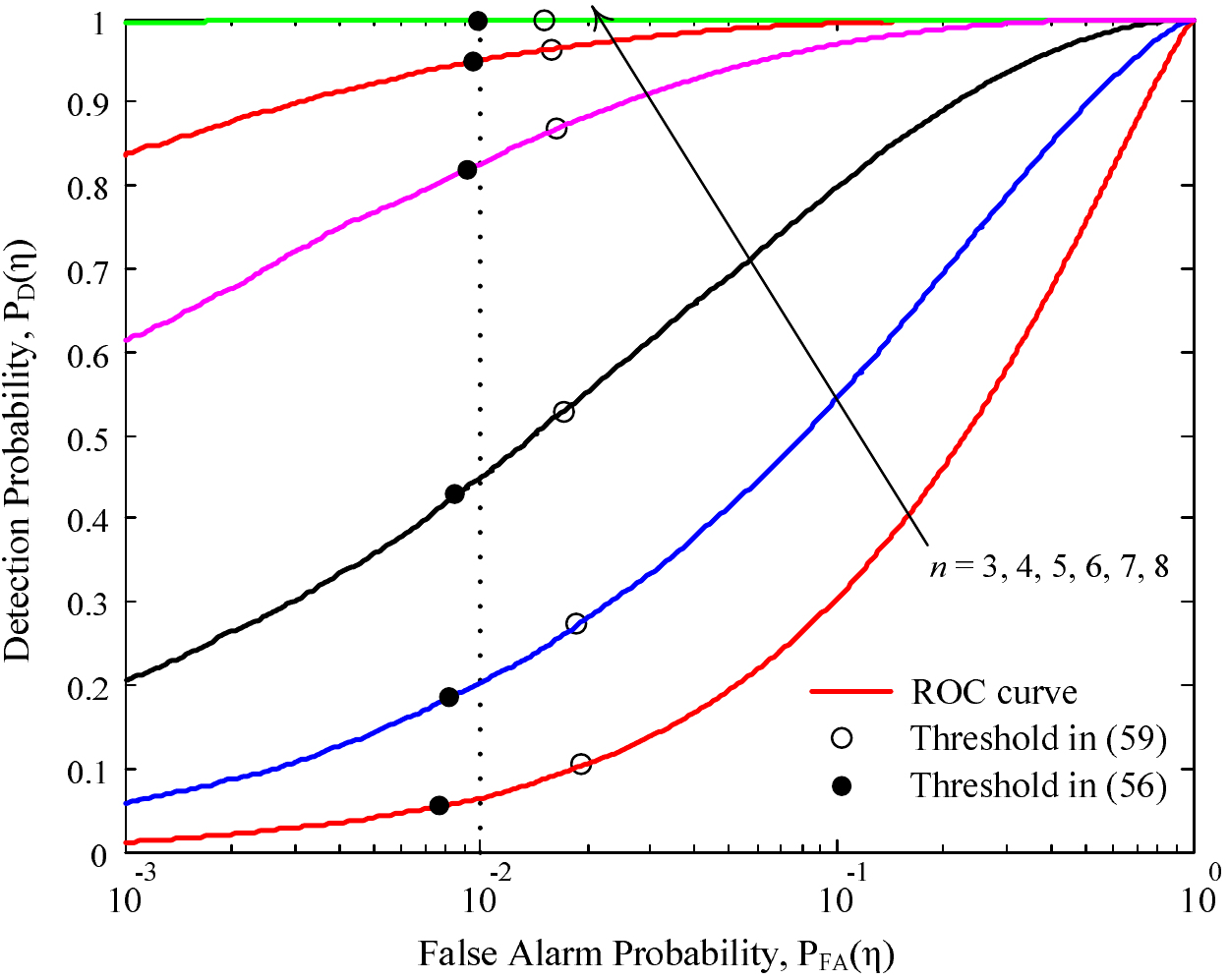}}
\caption{ROC curve (detection probability vs.\ false alarm probability) with $N_0=6$, $m = 5n$ and $k=n$. The vertical dotted line at a false alarm probability of $0.01$ is shown for convenience.}
\label{fig:ROC_testThreskappa01}
\end{figure}

\begin{figure}[t]
\centerline{\includegraphics[width=0.85\columnwidth]{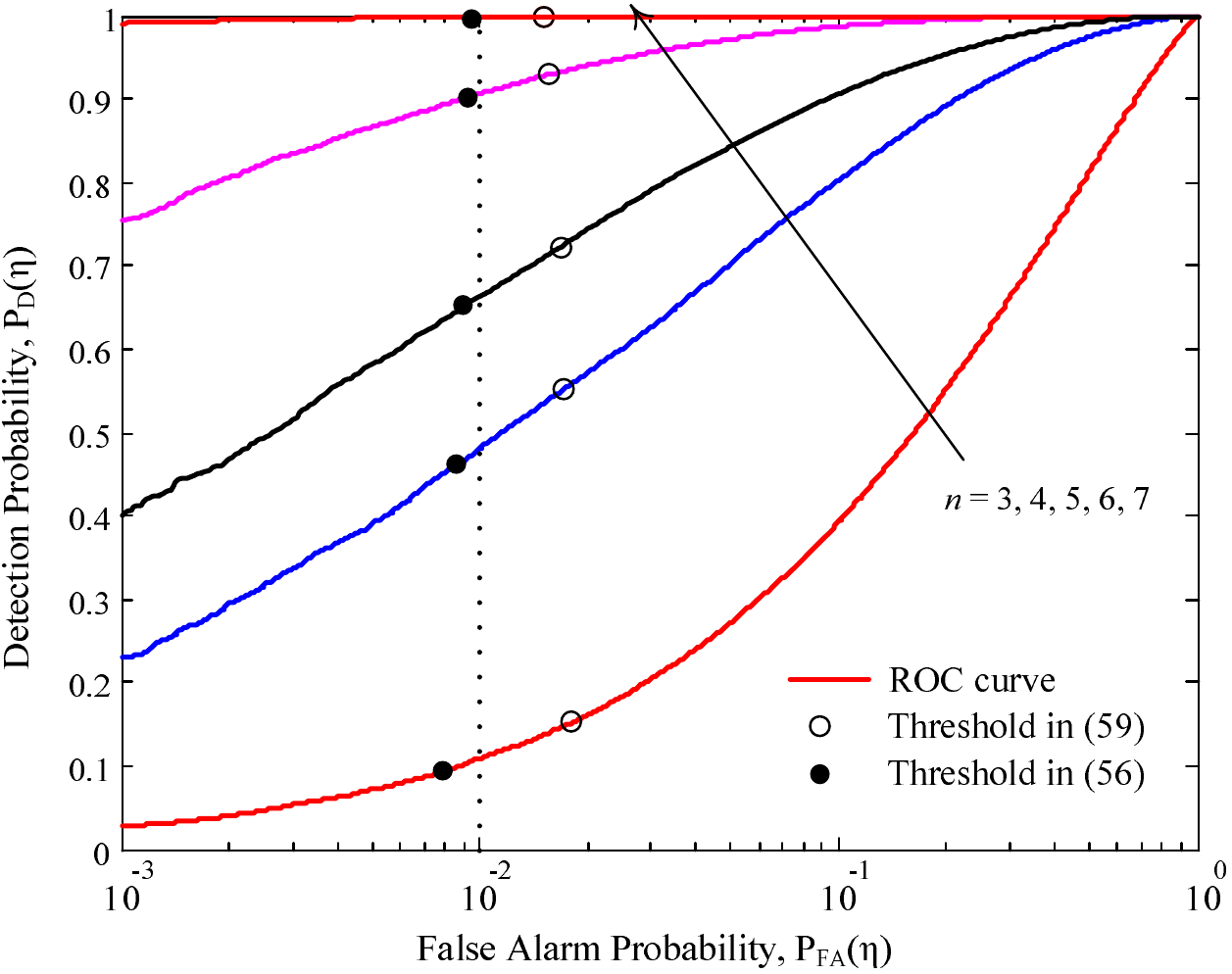}}
\caption{ROC curve (detection probability vs.\ false alarm probability) with $N_0=6$, $m = 7n$ and $k=n$. The vertical dotted line at a false alarm probability of $0.01$ is shown for convenience.}
\label{fig:ROC_testThreskappa02}
\end{figure}

The above design and ROC analysis is based on accurate representations of  ${\rm P}_{\rm FA}(\eta)$ and  ${\rm P}_{\rm D}(\eta)$, which have been numerically validated earlier in this paper. However, the design threshold needs to be computed by numerical search. In order to simplify the threshold design, we focus on our asymptotic results next.

For large $n$ and $m$, Theorem \ref{the:asymptotic_cdf} suggests that the false alarm probability can be approximated as
\begin{align}
 {\rm P}_{\rm FA}(\eta)= 1 - {\rm F}_{W_0}(\eta) \approx \frac{1}{2}\left(1-{\rm erf}\left(\frac{n\left(\eta - \bar{\mu}\right)}{\sqrt{2\bar{\sigma}^2}}\right)\right) \; ,
\end{align}
and thus the minimum threshold  which can satisfy a false alarm probability requirement of $\alpha_0$ can be approximated as
\begin{align}\label{eq:approx_threshold}
\eta_0 \approx \bar{\mu} +  \frac{\sqrt{2 \bar{\sigma}^2}}{n} {\rm erf}^{-1}\left(1-2\alpha_0 \right) \; .
\end{align}
A natural question then is whether this approximation is accurate enough for practical numbers of antennas and observations. To investigate this, we plot the ROC curve in Figs.\ \ref{fig:ROC_testThreskappa01} and \ref{fig:ROC_testThreskappa02} for $c=1/5$ and $c=1/7$ respectively.  The `ROC curve' is plotted using Monte Carlo simulations. The threshold computed via the Gaussian approximation in (\ref{eq:approx_threshold}) for a target false alarm probability of $\alpha_0=0.01$ is shown. For comparison, the threshold computed with the c.d.f. approximation function (through higher order cumulants) in (\ref{eq:approx_threshold_total}) is also shown.  We observe that the approximate threshold (\ref{eq:approx_threshold}) yields a false alarm probability slightly above the requirement $\alpha_0=0.01$, whilst this requirement is successfully met with the threshold in (\ref{eq:approx_threshold_total}). As expected, this loss in accuracy diminishes as $n$ (and consequently $m$) increases, which is in agreement with Theorem \ref{the:asymptotic_cdf}. Further, for satisfactory detection probabilities above $80\%$, the threshold in (\ref{eq:approx_threshold_total}) results in a false alarm probability that tightly meets the requirement $\alpha_0=0.01$.  
Moreover, we observe in Figs.\ \ref{fig:ROC_testThreskappa01} and \ref{fig:ROC_testThreskappa02} that decreasing $c$ for the same $n$ results in a higher detection probability, as more observations are utilized for detection.

\section{Conclusion}
\label{sec:conclusion}

Multiple-antenna signal detection has been addressed in cognitive radio networks with multiple primary user signals. By virtue of new closed-form expressions for the moments of the GLRT statistic, we have derived easy-to-compute and accurate expressions for the false alarm and detection probability. We have also proved that the GLRT statistic under hypothesis $\mathcal{H}_0$ converges to a Gaussian random variable when the number of antennas and observations grow large simultaneously. Further, the detection probability has been analyzed for a large number of primary user signals being no less than the number of receive antennas at the secondary user. Using results from large random matrix theory, we have shown that the (instantaneous) detection probability can be accurately approximated without explicit knowledge of the channel for a practical number of antennas. Leveraging our analytical results, simple design rules have been proposed to approximate the minimum detection threshold in order to achieve a desired false alarm probability.


\begin{appendix}
\subsection{Proof of Theorem \ref{the:semicorr_fullrank}}\label{app:semicorr_fullrank}

Let us first derive an expression for the moments of
\begin{align}
W = \frac{\frac{{\rm Tr}\left( \mathbf{X} \mathbf{X}^\dagger\right)}{n}}{ {\rm det}(\mathbf{X} \mathbf{X}^\dagger)^{\frac{1}{n}}} \; ,\quad \mathbf{X} \sim \mathcal{CN}_{n,m}\left(\mathbf{0}_{n,m}, \mathbf{R} \right) ,
\end{align}
where $\mathbf{R} \in \mathbb{C}^{n \times n}$ is a Hermitian positive definite matrix. Denote $0 \le \lambda_n \le\lambda_{n-1},\ldots, \le \lambda_1 < \infty$ as the ordered eigenvalues of $\mathbf{X} \mathbf{X}^\dagger$, which have a joint distribution\footnote{${\rm det}_n \hspace{-0.8mm} \left(g(i,j)\right)$ is the determinant of an $n \hspace{-0.5mm} \times \hspace{-0.5mm} n$ matrix with $(i,j)$th entry $g(i,j)$.} \cite{james1964}
\begin{align} \label{eq:jointPDFa}
 f_{\bLam}(\lambda_1,\ldots,\lambda_n) &=  \frac{(-1)^{\frac{n(n-1)}{2}}}{\prod_{\ell=1}^n (m-\ell)!}  {\rm det}_n \hspace{-0.7mm} \left(e^{-y_i^{-1} \lambda_j}\right) \notag \\
 & \times \frac{\left(\prod_{\ell=1}^n \lambda_\ell^{m-n} y_\ell^{-m}\right) {\rm det}_n \hspace{-0.7mm}  \left( \lambda_j^{i-1} \right)   }{\prod_{i<j} (y_i^{-1}-y_j^{-1})}
\end{align}
where $y_1, \ldots, y_n$ are the eigenvalues of $\mathbf{R}$. By denoting $\mathcal{D}=\{0 \le \lambda_n \le \ldots \le \lambda_1 < \infty\}$, we have
\begin{align}\label{eq:init_moment1}
& \mu_{W,p} \notag \\
&= \frac{1}{n^p} \int_{\mathcal{D}} \left(\frac{\sum_{i=1}^n \lambda_i}{\prod_{i=1}^n \lambda_i^{\frac{1}{n}}}\right)^{p}  f_\bLam(\lambda_1,\ldots,\lambda_n) {\rm d} \lambda_1,\ldots {\rm d} \lambda_n \notag \\
&= \frac{1}{n^p} \frac{ {\rm d}^p}{ {\rm d} \omega^p} \int_{\mathcal{D}} \frac{e^{\omega \sum_{i=1}^n \lambda_i} }{\prod_{i=1}^n \lambda_i^{\frac{p}{n}}}f_{\bLam}(\lambda_1,\ldots,\lambda_n) {\rm d} \lambda_1,\ldots {\rm d} \lambda_n \biggr|_{\omega=0} ,
\end{align}
where the second equality follows by noting that 
$\left. \frac{ {\rm d}^p}{ {\rm d} \omega^p} e^{\omega \psi} \right|_{\omega=0}= \psi^p $
and interchanging the integral and the derivative by virtue of Leibniz integral rule \cite[Eq. (3.3.7)]{Abramowitz1970}

Substituting (\ref{eq:jointPDFa}) into (\ref{eq:init_moment1}),
\begin{align} \label{eq:mu_I}
& \mu_{W,p}
= \frac{ (-1)^{\frac{n(n-1)}{2}} }{n^p \prod_{\ell=1}^n  (m-\ell)! }  \frac{\prod_{\ell=1}^n y_\ell^{-m}}{ \prod_{i<j} (y_i^{-1}-y_j^{-1})} \notag \\
& \frac{ {\rm d}^p}{ {\rm d} \omega^p} \hspace{-0.5mm}
\underbrace{ \int_{\mathcal{D}} \frac{e^{\omega \sum_{i=1}^n \lambda_i} {\rm det}_n \hspace{-0.7mm}  \left( \lambda_j^{i-1} \right) }{\prod_{i=1}^n \lambda_i^{\frac{p}{n}-m+n}}  {\rm det}_n \hspace{-0.7mm}  \left(\hspace{-0.5mm} e^{-\frac{\lambda_j}{y_i}} \right)
{\rm d} \lambda_1,\ldots {\rm d} \lambda_n }_{\mathcal{I}_1(\omega)} \biggr|_{\omega=0} \hspace{-0.7mm},
\end{align}
and the integral $\mathcal{I}_1(\omega)$, defined on the set $\{ \omega \in \mathbb{R} : y_j^{-1} -\omega > 0 \}$, can be evaluated as \cite[Eq. (51)]{chiani03b} 
\begin{align} \label{eq:I1_1}
\mathcal{I}_1(\omega) &=
{\rm det}_n \hspace{-0.7mm}\left(\int_0^\infty e^{-\lambda(y_j^{-1} -\omega)} \lambda^{m-n-\frac{p}{n}+i-1} {\rm d} \lambda\right) .
\end{align}
From \cite[Eq. (8.310)]{Gradshteyn1965}, we rewrite (\ref{eq:I1_1}) as
\begin{align} \label{eq:I1_2}
\mathcal{I}_1(\omega) &=
 {\rm det}_n \hspace{-0.7mm} \left(\frac{\Gamma\left(m-n-\frac{p}{n}+i\right)}{\left(y_j^{-1} -\omega\right)^{m-n-\frac{p}{n}+i}} \right) .
\end{align}

Substituting (\ref{eq:I1_2}) into (\ref{eq:mu_I}), followed by some algebraic manipulations, we obtain
\begin{align} \label{eq:init_moment3}
& \mu_{W,p} \notag \\
&=  \frac{ (-1)^{\frac{n(n-1)}{2}} }{n^p \prod_{\ell=1}^n  (m-\ell)! }  \frac{\prod_{i=1}^n y_i^{-m}  \Gamma\left(m-n-\frac{p}{n}+i\right)}{ \prod_{i<j} (y_i^{-1}-y_j^{-1})} \notag \\
&  \hspace{3mm} \times \frac{ {\rm d}^p}{ {\rm d} \omega^p} \frac{1}{\prod_{i=1}^n \left(y_i^{-1} -\omega\right)^{m-n-\frac{p}{n}}} {\rm det}_n \hspace{-0.8mm} \left(\frac{1}{\left(y_j^{-1} -\omega\right)^{i}} \right)   \biggr|_{\omega=0} \notag \\
&=  \frac{ (-1)^{\frac{n(n-1)}{2}} }{n^p } \prod_{i=1}^n \left(\frac{y_i^{-m} \Gamma\left(m-n-\frac{p}{n}+i\right)  }{ (m-i)!}\right) \mathcal{I}_2(\omega) \biggr|_{\omega=0} ,
\end{align} 
with
\begin{align}\label{eq:I_omega}
\mathcal{I}_2(\omega) &=\frac{ {\rm d}^p}{ {\rm d} \omega^p} \frac{1}{\prod_{i=1}^n \left(y_i^{-1} -\omega\right)^{m-\frac{p}{n}}} .
\end{align}
Applying Leibniz rule for differentiation \cite{Gradshteyn1965} gives
\begin{align}\label{eq:I_omega2}
& \mathcal{I}_2(\omega) \notag \\
  &= \sum_{k_1+\ldots + k_n=p} \binom{p}{k_1,\ldots,k_n} \prod_{i=1}^n \frac{ {\rm d}^{k_i}}{ {\rm d} \omega^{k_i}} \frac{1}{\left(y_i^{-1} -\omega\right)^{m-\frac{p}{n}}}  \notag \\
 &= \sum_{k_1+\ldots + k_n=p} \binom{p}{k_1,\ldots,k_n} \prod_{i=1}^n \prod_{j=1}^{k_i}  \frac{(-1)^{k_i}\left(m-\frac{p}{n}+j-1\right)}{\left(y_i^{-1} -\omega\right)^{m-\frac{p}{n}+k_i}} .
\end{align}
Substituting $\omega=0$ into (\ref{eq:I_omega2}), and the resultant expression into (\ref{eq:init_moment3}) followed by some algebraic manipulation, we obtain
\begin{align} \label{eq:moments_W}
\mu_{W,p}&=
\frac{ p! \prod_{i=1}^n y_i^{-\frac{p}{n}} }{n^p } \prod_{j=0}^{n-1}\frac{ \Gamma\left(m-n+1-\frac{p}{n}+j\right)  }{ \Gamma\left(m-n+1+j \right)} \notag \\
& \times \sum_{k_1+\ldots + k_n=p} \, \prod_{i=1}^n \frac{\Gamma\left(m-\frac{p}{n}+k_i\right) y_i^{k_i}}{\Gamma(k_i+1) \Gamma\left(m-\frac{p}{n}\right)} ,
\end{align}
which yields (\ref{eq:semicorr_fullrank}) for $W_1$ by setting $\mathbf{R} = \mathbf{H} \mathbf{H}^\dagger + \mathbf{I}_n N_0$, in which case, consequently, $N_0<y_1 \le y_2 \le \ldots\le y_n < \infty$. For the moments of $W_0$, (\ref{eq:moments_W}) simplifies to (\ref{eq:iid_fullrank}) after substituting $y_1=\ldots = y_n=N_0$.

\subsection{Proof of Proposition \ref{lem:expansion_iid}}
\label{app:asymp_moments}

We start by rewriting (\ref{eq:iid_fullrank}) as
\begin{align} \label{moments_barnes}
\mu_{W_0,p} = \frac{\prod _{j=1}^p(m n-j)}{n^p} \cdot \frac{G\left(m-\frac{p}{n}+1\right) \, G(m-n+1)}{G\left(m-n-\frac{p}{n}+1\right) \, G(m+1)} ,
\end{align}
where $G(\cdot)$ is the Barnes-$G$ function, which admits the following asymptotic expansion for large $z$ \cite{voros1987}:
\begin{align}\label{eq:barnesg}
\ln G(z+1) &= \frac{1}{12} - \ln \mathcal{A} + \frac{z}{2} \ln (2\pi) + \left(\frac{z^2}{2} - \frac{1}{12}\right) \ln z  \notag \\
& - \frac{3 z^2}{4}+ \sum_{k=1}^{\infty} \frac{B_{2k+2}}{4k(k+1) z^{2k}} 
\end{align}
where $\mathcal{A}$ is the Glaisher-Kinkelin constant \cite{voros1987} and $B_k$ is the Bernoulli number \cite[pp. 803]{Abramowitz1970}.

Taking the logarithm of (\ref{moments_barnes}) and noting that $m=n/c$,
\begin{align} \label{log_moments_barnes}
& \ln{\mu_{W_0,p}} = \sum _{j=1}^p \ln\left(\frac{n^2}{c}-j\right)  - p \ln{n} + \ln{G\left(\frac{n}{c}-\frac{p}{n}+1\right)} \nonumber \\
& + \ln{G\left(n\left(\frac{1}{c}-1\right)+1\right)} - \ln{G\left(\frac{n}{c}+1\right)} \nonumber \\
& - \ln{G\left(n(\frac{1}{c}-1)-\frac{p}{n}+1\right)} .
\end{align}
It is also convenient to note that, for large $n$,
\begin{align} \label{log_expansion}
\ln\left(\frac{n^2}{c}-j\right) = 2 \ln{n} -\ln{c} - \sum_{\ell=1}^{\infty} \frac{(c \, j)^{\ell}}{\ell \, n^{2\ell}} ,
\end{align}
and, therefore,
\begin{align} \label{log_expansion2}
\sum _{j=1}^p \ln\left(\frac{n^2}{c}-j\right) = 2p \ln{n} - p\ln{c} - \sum_{\ell=1}^{\infty} \frac{c^{\ell}}{\ell \, n^{2\ell}} H_{p,-\ell} ,
\end{align}
where $H_{a,b}$ are the Harmonic numbers. Using the expansions (\ref{log_expansion2}) and (\ref{eq:barnesg}) in (\ref{log_moments_barnes}), and after further algebra, we arrive at
\begin{align} \label{asymptoticLogMoments}
& \ln \mu_{W_0,p} \notag \\
&= \frac{3}{2}p  + \left(\frac{p}{c}-p-\frac{1}{2}\left(\frac{p}{n}\right)^2\right) \log(1-c) \nonumber \\
&-\sum _{\ell=1}^{\infty} \frac{c^{\ell}}{\ell \, n^{2\ell}} H_{p,-\ell}  -\left(\frac{1}{2}\left(\frac{n}{c}-\frac{p}{n}\right)^2-\frac{1}{12}\right) \sum _{\ell=1}^{\infty} \frac{(c p)^{\ell}}{\ell \, n^{2 \ell}} \nonumber \\
& + \left(\frac{1}{2}\left(\frac{n}{c}-n-\frac{p}{n}\right)^2-\frac{1}{12}\right)  \sum _{\ell=1}^{\infty} \frac{(c p)^{\ell}}{\ell \, (1-c)^{\ell} \, n^{2 \ell}} \nonumber \\
& +\sum _{k=1}^{\infty} \frac{B_{2k+2} \, c ^{2k}}{4k(k+1) \, n^{2k}} \sum _{r=1}^{\infty} \frac{(2k)_r (p c)^r}{r! \, n^{2r}} \left(1-(1-c)^{-2k-r} \right) .
\end{align}
Rearranging terms in (\ref{asymptoticLogMoments}) yields
\begin{align}\label{eq:logmellin_iid2}
 \ln \mu_{W_0,p}  = \sum_{q=0}^\infty \frac{ A_{p,q}(c)}{n^{2q}} ,
\end{align}
with coefficients $A_{p,q}(c)$ as given in Proposition \ref{lem:expansion_iid}.
From (\ref{eq:logmellin_iid2}), we have that
\begin{align} \label{muExpansionProd}
\mu_{W_0,p} &= \exp\left(\sum_{q=0}^\infty \frac{ A_{p,q}(c)}{n^{2q}}\right) \nonumber \\
&= e^{A_{p,0}(c)} \prod_{j=1}^{N} \sum_{r=0}^{N-j+1} \frac{1}{r!} \left( \frac{A_{p,j}(c)}{n^{2j}} \right)^r    + O\left(\frac{1}{n^{2(N+1)}}\right),
\end{align}
where the second equality is obtained after expanding $\exp{x}$ around $x=0$ with $N$ an arbitrary positive integer.
Finally, rearranging terms in (\ref{muExpansionProd}) yields the series in Proposition \ref{lem:expansion_iid}.


\subsection{The case c=1: asymptotic moments and cumulants}
\label{app:asymp_moments_c1}

The derivation steps are similar to those given in Appendix \ref{app:asymp_moments}. For $c=1$ (equivalently $m=n$), (\ref{moments_barnes}) specializes to
\begin{align} \label{moments_barnes_c1}
\mu_{W_0,p} = \frac{\prod _{j=1}^p(n^2-j)}{n^p}  \frac{G\left(n-\frac{p}{n}+1\right)}{G\left(1-\frac{p}{n}\right) \, G(n+1)} .
\end{align}
Taking logarithms at both sides of the equality,
\begin{align} \label{log_moments_barnes_c1}
 \ln{\mu_{W_0,p}} &= \sum _{j=1}^p \ln\left(n^2-j\right)  - p \ln{n} + \ln{G\left(n-\frac{p}{n}+1\right)} \nonumber \\
& - \ln{G\left(n+1\right)} - \ln{G\left(1-\frac{p}{n}\right)},
\end{align}
where, for $n$ large, the summation can be expanded as in (\ref{log_expansion2}) and the $\ln{G(\cdot)}$ terms can be expanded using (\ref{eq:barnesg}), with one exception: here, the term $\ln G\left(1-\frac{p}{n}\right)$ does not admit the expansion (\ref{eq:barnesg}) for $\log G(z+1)$, only valid for large $z$, and therefore we rely on the Taylor series expansion around $z=0$,
\begin{align} \label{log_barnes_taylor}
\log G(1-\frac{p}{n}) &= -\frac{p}{2n} \left( \ln{2\pi}-1 \right) - \frac{1}{2} \left( 1+ \mathcal{C} \right) \left( \frac{p}{n} \right)^2   \nonumber \\
& \quad - \sum_{k=3}^{\infty} \frac{p^k \zeta(k-1)}{k \, n^k} ,
\end{align}
where $\mathcal{C}$ is the Euler-Gamma constant and $\zeta(\cdot)$ is the Riemann-Zeta function.

After substitution of the corresponding expansions and tedious algebra, we arrive at   
\begin{align} \label{asymptoticLogMoments_c1}
& \log{\mu_{W_0,p}} = \frac{3}{2}p - \frac{p}{2n} + \frac{1}{2}\left(\frac{p}{n}\right)^2 \left(\mathcal{C}-\frac{1}{2}+\log{n}\right) \nonumber \\
& \quad - \frac{1}{2} \left(n^2+\left(\frac{p}{n}\right)^2-2p-\frac{1}{6} \right) \sum _{k=1}^{\infty}\frac{p^k}{k n^{2k}} - \sum _{k=1}^{\infty}\frac{H_{p,-k}}{k \, n^{2k}} \nonumber \\
& \quad +\sum _{k=3}^{\infty} \frac{\zeta(k-1)}{k} \left(\frac{p}{n}\right)^k  +\sum _{k=1}^{\infty} \frac{B_{2k+2}}{4k \, (k+1) \, n^{2k}} \sum _{r=1}^{\infty} \frac{\left(2k\right)_r \, p^r}{r! \, n^{2r}} .
\end{align}
Rearranging terms in (\ref{asymptoticLogMoments_c1}) results in the following corollary.

\begin{corollary}\label{lem:expansion_iid_c1}
With $c = 1$, the logarithm of $\mu_p$ admits
\begin{align}\label{eq:logmellin_iid_c1}
 \ln \mu_{W_0,p}  = \sum_{q=0}^\infty \frac{ A_{p,q}}{n^{q}}
\end{align}
where
\begin{align} \label{a0_c1}
A_{p,0} &= p \\ \label{a1_c1}
A_{p,1} &= -\frac{p}{2} \\ \label{a2_c1}
A_{p,2} &= \frac{1}{2}p^2(\mathcal{C}+\log{n})-\frac{5 p}{12} .
\end{align}
If $q>2$ and odd,
\begin{align} \label{aq_odd_c1}
A_{p,q} &= \frac{p^q}{q} \zeta(q-1) , 
\end{align}
whereas, for $q$ even, 
\begin{align}
 \label{aq_even_c1}
A_{p,q} &= \frac{p^q}{q} \zeta(q-1) - \frac{2 H_{p,\frac{-q}{2}} }{q}- \frac{2 p^{\frac{q}{2}+1}q}{q^2-4} \nonumber \\
& + \left(p+\frac{1}{12}\right) \frac{2p^{\frac{q}{2}}}{q}+\sum _{j=1}^{\frac{q}{2}-1}\frac{B_{2j+2} \, p^{\frac{q}{2}-j} \left(2j\right)_{\frac{q}{2}-j}}{4j \, (j+1)\left(\frac{q}{2}-j\right)!} .
\end{align}
\end{corollary}

Further, from (\ref{eq:logmellin_iid_c1}), we have that
\begin{align} \label{muExpansionProd_c1}
\mu_{W_0,p} &= \exp\left(\sum_{q=0}^\infty \frac{ A_{p,q}}{n^{q}}\right) \nonumber \\
&= e^{A_{p,0}(c)} \prod_{j=1}^{N} \sum_{r=0}^{N-j+1} \frac{1}{r!} \left( \frac{A_{p,j}}{n^{j}} \right)^r    + O\left(\frac{1}{n^{N+1}}\right) \; ,
\end{align}
where the second equality is obtained after expanding $\exp{x}$ around $x=0$ with $N$ an arbitrary positive integer.
Rearranging terms in (\ref{muExpansionProd_c1}) yields
\begin{align} \label{muExpansionSum_c1}
\mu_{W_0,p} = \sum_{j=0}^{\infty} \frac{\beta_{p,j}}{n^{j}}   \; ,
\end{align}
where $\beta_{p,0}=e^{A_{p,0}}$ and, for $j>0$,
\begin{align} \label{muExpansionCoeffs_c1}
\beta_{p,j} =  e^{A_{p,0}} \cdot \sum_{i_1+2i_2+\ldots+j i_j=j} \prod_{r=1}^j  \frac{ A_{p,r}^{i_r} } {i_r!}   \; ,
\end{align}
with $A_{p,q}$ given by (\ref{a0_c1})--(\ref{aq_even_c1}).

From (\ref{muExpansionSum_c1}) and the recursive relation between cumulants and moments (\ref{cumulants_moments}), we obtain the series expansion for the $p$th cumulant,
\begin{align} \label{kappaExpansion_c1}
\kappa_{W_0,1} &= \sum_{j=0}^{\infty} \frac{\alpha_{1,j}}{n^{j}}    \nonumber \\
\kappa_{W_0,p} &= \sum_{j=0}^{\infty} \frac{\alpha_{p,j}}{n^{p+j}}   \; , \,\, p>1 ,
\end{align}      
where $\alpha_{1,j}=\beta_{1,j}$ and  the rest of coefficients ($p>1$) obtained recursively as
\begin{align} \label{alphaRecursion_c1}
\alpha_{p,k} &= \beta_{p,p+k} - \sum _{j=0}^{p+k} \alpha_{1,j} \beta_{p-1,p+k-j} \nonumber \\
&- \sum _{r=2}^{p-1} \binom{p-1}{r-1} \sum _{j=0}^{p-r+k} \alpha_{r,j} \beta_{p-r,p-r+k-j} \,\, ,
\end{align}
with $\beta_{p,j}$ given by (\ref{muExpansionCoeffs_c1}).

Leveraging the above expressions, the leading-order term of the $p$th cumulant is found to be
\begin{align}
\alpha_{1,0} &= e \\
\alpha_{2,0} &= e^2 \, (\mathcal{C}+\log{n}) \\
\alpha_{p,0} &= e^p \, (p-1)! \, \zeta(p-1) , \quad p>2.
\end{align}


\subsection{Proof of Theorem \ref{the:asymptotic_cdf}}\label{app:asymptotic_cdf}

Through the invariance and homogeneity property of cumulants, the $p$th cumulant of $n(W_0 - \bar\mu)$, for $p\ge2$, can be written as $n^p \kappa_{W_0,p}$. From (\ref{kappaExpansion}), we thus observe that for $p\geq 3$, $\lim_{n \to \infty} n^p \kappa_{W_0,p} =0$, and thus $\lim_{n \to \infty} n(W_0 - \bar\mu)$ follows a Gaussian distribution, with zero mean and variance given by $\lim_{n \to \infty} n^2 \kappa_{W_0,2}$ which is obtained from $\kappa_{W_0,2}$ given by (\ref{k2}).

\subsection{Proof of Proposition \ref{lemm:semi_variance}}\label{app:semi_variance}

We start by expressing the moments of $W_1$ in terms of those of $W_0$, which can be expanded from Proposition \ref{lem:expansion_iid}.

From (\ref{eq:semicorr_fullrank}) and (\ref{eq:iid_fullrank}), we can write for $p=1$,
\begin{align}
\mu_{W_1,1} &= \left( \prod_{i=1}^n y_i^{\frac{1}{n}} \right) \frac{1}{n} \mu_{W_0,1} \sum_{j=1}^n \frac{1}{y_j} \notag \\ 
&= \Psi_1 \Psi_2 \mu_{W_0,1}  .
\end{align}
Leveraging Proposition \ref{lem:expansion_iid} yields
\begin{align} \label{eq:m1W1}
\mu_{W_1,1}
&= \Psi_1 \Psi_2 e^{A_{1,0}(c)} \notag \\
&  \times \left( 1 + \frac{1}{n^2} A_{1,1}(c) + \frac{1}{n^4} \left( \frac{A_{1,1}^2(c)}{2} + A_{1,2}(c)  \right)  \right) \notag \\
& +O\left(\frac{1}{n^6}\right) ,
\end{align}
which gives (\ref{eq:asympMeanW1}) after taking the leading-order term and substituting $A_{1,0}$ with (\ref{a0}). 


For $p=2$, we have
\begin{align}
& \mu_{W_1,2}   \notag \\
&=\frac{ \prod_{i=1}^n y_i^{\frac{2}{n}} }{n^2 } \frac{\mu_{W_0,2}}{\left( m -\frac{1}{n} \right)\left( m -\frac{2}{n} \right)}
 \left( \left(m-\frac{2}{n} \right) \left(m-\frac{2}{n}+1 \right) \right. \notag \\
& \times \left. \sum_{i=1}^n \frac{1}{y_i^{2}}    + 2 \left(m-\frac{2}{n}\right)^2\sum_{i=2}^n \sum_{j=1}^{i-1} \frac{1}{y_i y_j} \right)  \notag \\ 
&= \Psi_1^2 \mu_{W_0,2} \left( \Psi_2^2 + (\Psi_3 -\Psi_2^2) \frac{1}{n m \left(1-\frac{1}{n m} \right)} \right) ,
\end{align}
where the second equality follows from algebraic manipulations by noting that
\begin{align} \label{relationM2}
\left( \sum_{i=1}^n \frac{1}{y_i} \right)^2 = \sum_{i=1}^n \frac{1}{y_i^{2}} + 2 \sum_{i=2}^n \sum_{j=1}^{i-1} \frac{1}{y_i y_j} .
\end{align}
Using the expansion
\begin{align} \label{fracExpansion}
\frac{1}{n m \left(1-\frac{a}{n m} \right)} = \sum_{k=1}^{\infty} \frac{c^k a^{k-1}}{n^{2k}}
\end{align}
with $m=n/c$ and Proposition \ref{lem:expansion_iid} we arrive at
\begin{align} \label{eq:m2W1}
\mu_{W_1,2} &= \Psi_1^2 e^{A_{2,0}(c)} \left( \Psi_2^2 + \frac{1}{n^2} \left( \Psi_2^2 (A_{2,1}(c)-c) + \Psi_3 c \right) \right. \notag \\
& \left. + \frac{1}{n^4} \left( \Psi_2^2 \left(\frac{A_{2,1}^2(c)}{2} + A_{2,2}(c) \right) \right. \right. \notag \\
& \left. \left. +(\Psi_3-\Psi_2^2) (A_{2,1}(c)+c^2) \right) \right) 
+O\left(\frac{1}{n^6}\right) .
\end{align}

Now, the variance $ \sigma_{W_1}^2 = \mu_{W_1,2} - \mu_{W_1,1}^2$ is obtained using (\ref{eq:m1W1}) and (\ref{eq:m2W1}) as
\begin{align}
\sigma_{W_1}^2 &= \frac{1}{n^2} \Psi_1^2 e^{A_{2,0}(c)} \notag \\
& \left( \Psi_2^2 (A_{2,1}(c)-c) + \Psi_3 c - 2 \Psi_2^2 A_{1,1}(c)   \right) +O\left(\frac{1}{n^4}\right) ,
\end{align}
which yields (\ref{eq:varW1}) upon substituting $A_{2,0}(c)$, $A_{2,1}(c)$, and $A_{1,1}(c)$ with their respective values given in Proposition \ref{lem:expansion_iid}.

For the third cumulant, we first compute the third moment from (\ref{eq:semicorr_fullrank}). For $p=3$ we have
\begin{align} \label{m3W1}
\mu_{W_1,3}
&=\frac{ 6 \prod_{i=1}^n y_i^{\frac{2}{n}} }{n^3 } \frac{\mu_{W_0,3}}{\left( m -\frac{1}{n} \right)\left( m -\frac{2}{n} \right) }
 \notag \\
 &\left( \frac{1}{6}  \left(m-\frac{3}{n} +2 \right) \left(m-\frac{3}{n}+1 \right)  \sum_{i=1}^n \frac{1}{y_i^{3}} \right. \notag \\
& \left.     + \frac{1}{2} \left(m-\frac{3}{n}+1\right) \left(m-\frac{3}{n}\right) \sum_{ 
\begin{subarray}{c} i,j=1,\ldots,n  \\
i \neq j  \end{subarray}
} \frac{1}{y_i^2 y_j} \right. \notag \\
& \left. + \left(m-\frac{3}{n}\right)^2 \sum_{\begin{subarray}{c} i,j,k=1,\ldots,n \\   i \neq j \neq k \end{subarray}} \frac{1}{y_i y_j y_k}  \right)  .
\end{align}
Noting that
\begin{align}
 & \left( \sum_{i=1}^n \frac{1}{y_i^{3}} \right)^3 =  \sum_{i=1}^n \frac{1}{y_i^{3}} + 3 \hspace{-3mm} \sum_{\begin{subarray}{c}  i,j=1,\ldots,n \\  i \neq j \end{subarray}} \frac{1}{y_i^2 y_j}  + 6  \hspace{-3mm}  \sum_{\begin{subarray}{c}  i,j,k=1,\ldots,n \\ i \neq j \neq k \end{subarray}} \frac{1}{y_i y_j y_k} ,
\end{align}
and
\begin{align}
 \sum_{i=1}^n \frac{1}{y_i}   \sum_{k=1}^n \frac{1}{y_k^2} =  \sum_{k=1}^n \frac{1}{y_k^3}
   \sum_{\begin{subarray}{c} i,j=1,\ldots,n \\  i \neq j \end{subarray}} \frac{1}{y_i^2 y_j}  ,
\end{align}
we can rewrite (\ref{m3W1}) as
\begin{align}
\mu_{W_1,3}
&= \Psi_1^3 \frac{\mu_{W_0,3}}{\left( m -\frac{1}{n} \right)\left( m -\frac{2}{n} \right) } \notag \\
& \left( \Psi_2 \Psi_3 \frac{3}{n} \left( m - \frac{3}{n}\right) + \Psi_2^3 \left( m - \frac{3}{n}\right)^2 + \Psi_4 \frac{2}{n^2} \right) \notag \\
&=  \Psi_1^3 \mu_{W_0,3} \left( \Psi_2^3 + \Psi_2^3   \frac{1}{n m \left(1-\frac{2}{n m} \right)}  \right. \notag \\
& \left.   + \left( 3 \Psi_2 \Psi_3 -4 \Psi_2^3 \right) \frac{1}{n m \left(1-\frac{1}{n m} \right)} \right. \notag \\
& \left. + \left( 2\Psi_4 -3\Psi_2\Psi_3 \right) \frac{1}{n^2 m^2 \left(1-\frac{1}{n m} \right) \left(1-\frac{2}{n m} \right)}  \right) .
\end{align}
Then, using (\ref{fracExpansion}) together with the expansion for $\mu_{W_0,3}$ given in Proposition \ref{lem:expansion_iid}, we arrive at
\begin{align} \label{eq:m3W1}
& \mu_{W_1,3} \notag \\
&=  \Psi_1^3 e^{A_{3,0}(c)} \left( \Psi_2^3 + \frac{1}{n^2} \left( \Psi_2^3 A_{3,1}(c) + 3 \Psi_2 \Psi_3 c - 3 \Psi_2^3 c \right) \right. \notag \\
& \left. + \frac{1}{n^4}\left(\Psi_2^3 \left(\frac{ A_{3,1}^2(c) }{2}+ A_{3,2}(c) \right)+3 \Psi_2 \Psi_3 c A_{3,1}(c) \right. \right. \notag \\
& \left. \left. - 2 \Psi_2^3 c^2 - 3 \Psi_2^3 c A_{3,1}(c) + 2\Psi_4 c^2 \right)\right) +O\left(\frac{1}{n^6}\right).
\end{align}
Finally, the third cumulant,
\begin{align}
 \kappa_{W_1,3} = \mu_{W_1,3} - 3 \mu_{W_1,2} \mu_{W_1,1} + 2 \mu_{W_1,1}^3 ,
\end{align}
is obtained using (\ref{eq:m1W1}), (\ref{eq:m2W1}), and (\ref{eq:m3W1}), resulting in
\begin{align}
\kappa_{W_1,3} &= \frac{1}{n^4} e^{A_{3,0}(c)} \Psi_1^3  \left( \Psi_2^3 \left( \frac{A_{3,1}^2(c)}{2} + A_{3,2}(c) - 3 c A_{3,1}(c) \right. \right. \notag \\
& \left. \left.  - 
      3 A_{1,1}(c) (A_{2,1}(c) - c) - 3 \frac{A_{2,1}^2(c)}{2} -3 A_{2,2}(c) \right. \right. \notag \\
      & \left. \left. + 
      3 A_{2,1}(c) c + c^2 + \frac{15}{2} A_{1,1}^2(c) + 3 A_{1,2}(c) \right) \right. \notag \\
      & \left. + 
   \Psi_2 \Psi_3 3 c (A_{3,1}(c) - A_{1,1}(c) - A_{2,1}(c) - c) + 
   \Psi_4 2 c^2 \right) ,
\end{align}
which finally yields (\ref{eq:k3W1}) upon substituting $A_{p,q}(c)$ with the expressions given by Proposition \ref{lem:expansion_iid}, and further simplifications.

%
%
%

\end{appendix}

\bibliographystyle{IEEEtran}
\bibliography{IEEEabrv,refs}
\end{document}